\documentclass[10pt,a4paper]{amsart}
\usepackage[utf8]{inputenc}
\usepackage[T1]{fontenc}
\usepackage{amsmath}
\usepackage{amsfonts}
\usepackage{amssymb}
\usepackage[english]{babel}
\usepackage{tabularx}
\usepackage{mathtools}
\usepackage{adjustbox}
\usepackage{graphics}
\usepackage{hyperref}
\usepackage{xcolor}
\usepackage{color}
\usepackage{tikz}
\usetikzlibrary{calc}
\numberwithin{equation}{section}
\newtheorem{thm}{Theorem}[section]
\newtheorem{pr}[thm]{Proposition}
\newtheorem{lm}[thm]{Lemma}
\newtheorem{re}[thm]{Remark}
\newtheorem{df}[thm]{Definition}
\newtheorem{ex}[thm]{Example}
\newtheorem{cor}[thm]{Corollary}

\newtheorem{con}[thm]{Conjecture}

\newcommand{\lcm}{\text{lcm}}

\newcommand{\stirlingone}[2]{\genfrac{[}{]}{0pt}{}{#1}{#2}}

\newcommand{\floor}[1]{\left\lfloor #1 \right\rfloor}
\newcommand{\flce}[1]{\left\lfloor #1 \right\rceil}

\theoremstyle{remark}

\makeatletter
\let\@@pmod\pmod
\DeclareRobustCommand{\pmod}{\@ifstar\@pmods\@@pmod}
\def\@pmods#1{\mkern4mu({\operator@font mod}\mkern 6mu#1)}
\makeatother

\title[Some inequalities for the restricted partition function $p_\mathcal{A}(n,k)$]{Log-concavity of the restricted partition function $p_\mathcal{A}(n,k)$ and the new Bessenrodt-Ono type inequality}
\author{Krystian Gajdzica}
\address{Institute of Mathematics \\
	Faculty of Mathematics and Computer Science \\
	Jagiellonian University in Cracow
}
\email{krystian.gajdzica@im.uj.edu.pl}



\begin{document}

\setlength{\parindent}{10mm}
\maketitle

\begin{abstract}
    Let $\mathcal{A}=(a_i)_{i=1}^\infty$ be a non-decreasing sequence of positive integers and let  $k\in\mathbb{N}_+$ be fixed. The function $p_\mathcal{A}(n,k)$ counts the number of partitions of $n$ with parts in the multiset $\{a_1,a_2,\ldots,a_k\}$. We find out a new type of Bessenrodt-Ono inequality for the function $p_\mathcal{A}(n,k)$. Further, we discover when and under what conditions on $k$, $\{a_1,a_2,\ldots,a_k\}$ and $N\in\mathbb{N}_+$, the sequence $\left(p_\mathcal{A}(n,k)\right)_{n=N}^\infty$ is $\log$-concave. Our proofs are based on the asymptotic behavior of $p_\mathcal{A}(n,k)$ --- in particular, we apply the results of Netto and P\'olya-Szeg\H{o} as well as the Almkavist's estimation.
    
\end{abstract}

\section{Introduction}

Let $n$ be a non-negative integer. By a partition $\lambda$ of $n$, we mean every non-increasing sequence of positive integers $\lambda
_1,\lambda_2,\ldots,\lambda_j$ such that
$$n=\lambda_1+\lambda_2+\cdots+\lambda_j.$$
Elements $\lambda_i$ are called the parts of the partition $\lambda$. Now, one can ask --- how many such sequences do there exist for a given parameter $n$? Therefore, we define the partition function $p(n)$ which enumerates all possible partitions of $n$. In particular, there are five partitions of $n=4$, namely, $(4)$, $(3,1)$, $(2,2)$, $(2,1,1)$ and $(1,1,1,1)$. Hence, $p(4)=5.$ Clearly, $p(n)=0$ when $n$ is negative, and $p(0)=1$, because the empty sequence is the only one in this case. In $1748$ Euler discovered the generating function for $p(n)$, that is
\begin{align*}
\sum_{n=0}^\infty p(n)x^n=\prod_{i=1}^\infty\frac{1}{1-x^i}.
\end{align*}
It is worth underlying that the partition function plays a crucial role in many parts of mathematics; and for centuries, its properties have been investigated from both combinatorial and number-theoretical points of view. Thus, there is an abundance of literature devoted to the theory of partitions. The general introduction to the topic might be found, for example, in Andrews' books \cite{GA2, GA1}. 

A few years ago there began intense research related to multiplicative behavior of the partition function. The first work of this kind is due to DeSalvo and Pak \cite{DSP}, who reproved the result (obtained by Nicolas \cite{N}) that the sequence $p(n)$ is $\log$-concave for all $n>25$:
$$p^2(n)>p(n-1)p(n+1).$$
Moreover, they also resolved two related conjectures by Chen and one by Sun (for more details, see \cite[pp. 117–121]{Chen} and \cite{Sun}, respectively), namely, they performed the following three results.
\begin{thm}
For all $n>6$, we have
\begin{align*}
    \frac{p(n-1)}{p(n)}\left(1+\frac{240}{(24n)^{3/2}}\right)>\frac{p(n)}{p(n+1)}.
\end{align*}
\end{thm}
\begin{thm}
For all $n>m>1$, we have
\begin{align*}
    p^2(n)>p(n+m)p(n-m).
\end{align*}
\end{thm}
\begin{thm}
The sequence $\frac{p(n)}{n}$ is log-concave for all $n\geqslant31.$
\end{thm}
\noindent Their proofs are based on Rademacher type estimates \cite{HR} by Lehmer \cite{Lehmer}. 

Afterward, Chern, Fu and Tang \cite{CFT} stated another conjecture related to the so-called $k$-colored partition function, where $k\in\mathbb{N}_+$. The $k$-colored partition function $p_{-k}(n)$ counts all possible partitions of $n$ in which every part may appear in $k$ distinct colors. The generating function for $p_{-k}(n)$ satisfies
$$\sum_{n=0}^\infty p_{-k}(n)x^n=\prod_{i=1}^\infty\frac{1}{(1-x^i)^k}.$$
They postulated that for every triples of positive integers $(k,n,l)$ such that $k\geqslant2$ and $n>l$, if $(k,n,l)\neq(2,6,4)$, then
$$p_{-k}(n-1)p_{-k}(l+1)\geqslant p_{-k}(n)p_{-k}(l).$$ 
Heim and Neuhaser \cite{HN1} generalized the above conjecture for all $k\in\mathbb{R}_{\geqslant2}$ --- that is the case of so-called D'Arcais polynomials or Nekrasov-Okounkov polynomials
 (for more details, see \cite{HAN,NO}). The aforementioned conjectures were completely resolved in the first case and partially in the second one by Bringmann, Kane, Rolen, and Tripp \cite{BKRT}. The authors also performed an asymptotic approach to the problems.

However, there are another intriguing multiplicative inequalities in the theory of partitions. In 2016 Bessenrodt and Ono \cite{B-O} demonstrated the following theorem.

\begin{thm}
For all integers $a,b>1$ such that $a+b>9$, we have
$$p(a)p(b)>p(a+b).$$
\end{thm}
Their proof (similarly to those mentioned above) is based on asymptotic estimates. Actually, there are also other proofs: the combinatorial proof given by A., G. and M. \cite{AGM}, and the proof by induction presented by H. and N. \cite{HN2}.

Many mathematicians have extended the Bessenrodt-Ono inequality for other variations of the partition function. Chern, Fu and Tang \cite{CFT} showed similar property for the $k$-colored partition function $p_{-k}(n)$. Heim, Neuhauser and Tr\"{o}ger \cite{HNT} generalized the aforementioned results to D'Arcais polynomials. Moreover, Beckwith and Bessenrodt \cite{BB} found out an analogous inequality for the so-called $k$-regular partition function. Hou and Jagadeesan \cite{HJ} obtained such an identity for the number of partitions with rank in a given residue class modulo $3$. Males \cite{M} extended their result for any residue class modulo $t\geqslant2$. Further, Dawsey and Masri \cite{DM} discovered the Bessenrodt-Ono type of inequality for the Andrews $spt$-function. More recently Heim, Neuhauser and Tr\"{o}ger \cite{BNT2} investigated, in this regard, the plane partition function and its polynomization. 

This sequel, on the other hand, is devoted to research on multiplicative properties of the restricted partition function $p_\mathcal{A}(n,k)$ --- it turns out that $p_\mathcal{A}(n,k)$ usually (but not always) fulfills analogous inequalities to these mentioned above. 

Let $\mathcal{A}=(a_i)_{i=1}^\infty$ be an increasing (a non-decreasing) sequence of positive integers, and let $k\in\mathbb{N}_+$ be fixed. A restricted partition is a partition in which every part belongs to the set (multiset) $\{a_1,a_2,\ldots,a_k\}$. Furthermore, two restricted partitions are considered the same if there is only a difference in the order of their parts. The (multicolor) restricted partition function $p_\mathcal{A}(n,k)$ enumerates all possible such partitions of $n$. Naturally $p_\mathcal{A}(n,k)=0$, if $n$ is negative, and $p_\mathcal{A}(0,k)=1$. We use the same notation to the restricted partition function and the multicolor restricted partition function, mainly because the second of them is just a generalization of the first one. The generating function for $p_\mathcal{A}(n,k)$ takes the form
\begin{equation}\label{GF}
    \sum_{n=0}^\infty p_\mathcal{A}(n,k)x^n=\prod_{i=1}^k\frac{1}{1-x^{a_i}}.
\end{equation}
For instance, if $\mathcal{A}=(2^{i-1})_{i=1}^\infty$, then we have exactly $6$ restricted partitions of $n=6$ for $k=3$ --- that are $(4,2), (4,1,1), (2,2,2), (2,2,1,1), (2,1,1,1,1)$ and $(1,1,1,1,1,1)$ --- in other words $p_\mathcal{A}(6,3)=6$. There is a wealth of literature related to both arithmetic properties (for example, see \cite{K,LS,RS,MU}) and asymptotic behavior (see, e.g., \cite{GA, CN, DV, MBN}) of $p_\mathcal{A}(n,k)$.

In this paper we find out the Bessenrodt-Ono type of inequality in the cases of both the restricted partition function and the multicolor restricted partition function. Moreover, we investigate the $\log$-concavity of the sequence $\left(p_\mathcal{A}(n,k)\right)_{n=1}^\infty$. Our approach is based on asymptotic estimates, namely, we apply the results of Netto \cite{Netto} and P\'olya-Szeg\H{o} \cite{PS} as well as the Almkavist's estimation for $p_\mathcal{A}(n,k)$ \cite{GA}. Let us also note that the formula obtained by Cimpoeaş and Nicolae \cite{CN} for $p_\mathcal{A}(n,k)$ is crucial in the investigation.


This paper is organized as follows. In Sec. 2 we introduce necessary concepts, properties and tools which are systematically used in the sequel. Sec. 3 and Sec. 4 are devoted only to the restricted partition function. The first of them deals with the Bessenrodt-Ono type of inequality, while the second one with the $\log$-concavity of the sequence $\left(p_\mathcal{A}(n,k)\right)_{n=1}^\infty$. Finally, Sec. 5 extends the previously obtained results to the multicolor restricted partition function.


\section{Preliminaries}
At the beginning, let us introduce some notations and conventions. By $\mathbb{N}$ and $\mathbb{N}_+$, we denote the set of non-negative integers and the set of positive integers, respectively. Moreover, for a given positive integer $k$ we also put $\mathbb{N}_{\geqslant k}=\{k,k+1,k+2,\ldots\}$.

Next, let $\mathcal{A}=(a_i)_{i=1}^\infty$ be an increasing sequence of positive integers, and let $k\in\mathbb{N}_+$ and $n\in\mathbb{N}$ be fixed. The restricted partition function $p_\mathcal{A}(n,k)$ counts the number of partitions of $n$ with parts in $\{a_1,a_2,\ldots,a_k\}$. As usual, we extend the definition of $p_\mathcal{A}(n,k)$ to all integers $n$ by setting  $p_\mathcal{A}(0,k)=1$, and $p_\mathcal{A}(n,k)=0$ for negative values of $n$. There is a well-known recurrence formula for the restricted partition function, which can be easily obtained, for instance, by manipulation of the equation (\ref{GF}). We perform a simple reasoning in a combinatorial manner.
\begin{pr}\label{Pr2.1}
The recurrence formula
\begin{equation}\label{Rf}
p_\mathcal{A}(n,k)=p_\mathcal{A}(n-a_k,k)+p_\mathcal{A}(n,k-1)
\end{equation}
holds for all $n\in\mathbb{N}$ and $k\in\mathbb{N}_{\geqslant2}$. For $k=1$, we have
\begin{align*}
	        p_\mathcal{A}(n,1)=\begin{cases}
        0, & \text{if } a_1\nmid n,\\
        1, & \text{if } a_1\mid n.
        \end{cases}
	    \end{align*}
\end{pr}
\begin{proof}
If $k=1$, then $p_\mathcal{A}(n,k)=1$ if and only if $n$ is a multiple of $a_1$ otherwise $p_\mathcal{A}(n,k)=0$; and the equality above is clear. 

Now, let us assume that $k>1$. In order to determine the value of $p_\mathcal{A}(n,k)$, we may consider two alternatives, namely, either we take $a_k$ as a part or not. If we do, then we actually deal with $p_\mathcal{A}(n-a_k,k)$ --- since we count all the restricted partitions of $n$ in which there is at least one part $a_k$, therefore we may subtract $a_k$ from $n$ and calculate all the partitions of $n-a_k$ with parts in $\{a_1,a_2,\ldots,a_k\}$. On the other hand, if we do not take $a_k$ as a part, then we just consider the value of $p_\mathcal{A}(n,k-1)$ --- because we enumerate all the partitions of $n$ with parts in $\{a_1,a_2,\ldots,a_{k-1}\}$. Hence, the recurrence formula (\ref{Rf}) holds.
\end{proof}
There is an immediate useful consequence of the aforementioned result.

\begin{cor}
For all $k\in\mathbb{N}_{\geqslant2}$ and $n\in\mathbb{N}$, we have 
\begin{equation}\label{Rf2}
p_\mathcal{A}(n,k)=\sum_{i=0}^{\floor{\frac{n}{a_k}}} p_\mathcal{A}(n-ia_k,k-1).
\end{equation}
\end{cor}
\begin{proof}
For fixed parameters $k$ and $n$ as above, it is enough to systematically use the formula (\ref{Rf}) as follows:
\begin{align*}
p_\mathcal{A}(n,k)&=p_\mathcal{A}(n-a_k,k)+p_\mathcal{A}(n,k-1)\\
&=p_\mathcal{A}(n-2a_k,k)+p_\mathcal{A}(n-a_k,k-1)+p_\mathcal{A}(n,k-1)\\
&\vdots\\
&=p_\mathcal{A}\left(n-\floor{\frac{n}{a_k}}a_k,k-1\right)+\cdots+p_\mathcal{A}(n-a_k,k-1)+p_\mathcal{A}(n,k-1).
\end{align*}
In conclusion, we get
\begin{align*}
p_\mathcal{A}(n,k)=\sum_{i=0}^{\floor{\frac{n}{a_k}}} p_\mathcal{A}(n-ia_k,k-1),
\end{align*}
as required.
\end{proof}

In the paper we will use two well-known asymptotic results related to the function $p_\mathcal{A}(n,k)$. The first one was performed, for instance, by Netto \cite{Netto} or P\'{o}lya-Szeg\H{o} \cite{PS} (their proofs are based on the partial fraction decomposition). It is worth mentioning that there is also a proof by induction on $k$ due to Nathanson \cite{MBN}.

\begin{thm}\label{2.3}
Let $\mathcal{A}=\left(a_i\right)_{i=1}^\infty$ be an increasing sequence of positive integers, and let $k\in\mathbb{N}_{\geqslant2}$ be fixed. If $\gcd(a_1,a_2,\ldots,a_k)=1$, then
\begin{equation}
p_\mathcal{A}(n,k)=\frac{n^{k-1}}{(k-1)!\prod_{i=1}^ka_i}+O(n^{k-2}).
\end{equation}
\end{thm}
On the other hand, the second result is a more refined asymptotic formula for $p_\mathcal{A}(n,k)$. It was obtained, for instance, by Almkvist \cite{GA}, Beck, Gessel and Komatsu \cite{BGK} or Israilov \cite{I}. Almkvist \cite{GA} performed this in an elegant way. In order to state the theorem in his style, let us introduce some additional notation, namely, we define symmetric polynomials $\sigma_i(x_1,x_2,\ldots,x_k)$ by the power series expansion
\begin{align*}
\prod_{i=1}^k\frac{x_it/2}{\sinh(x_it/2)}=\sum_{m=0}^\infty\sigma_m(x_1,x_2,\ldots,x_k)t^m.
\end{align*} 
It turns out that the polynomials $\sigma_i$ are directly connected with the so-called polynomial part of $p_\mathcal{A}(n,k)$, as we see in the following.
\begin{thm}\label{2.4}
Let $\mathcal{A}=\left(a_i\right)_{i=1}^\infty$ be an increasing sequence of positive integers, and let $k\in\mathbb{N}_{\geqslant2}$ be fixed. For a given integer $j\in\{1,2,\ldots,k\}$, if $\gcd A=1$ for each $j$-element subset ($j$-subset) $A$ of $\{a_1,a_2,\ldots,a_k\}$ and $\sigma=a_1+a_2+\cdots+a_k$, then
\begin{equation}
p_\mathcal{A}(n,k)=\frac{1}{\prod_{i=1}^ka_i}\sum_{i=0}^{k-j}\sigma_i(a_1,a_2,\ldots,a_k)\frac{(n+\sigma/2)^{k-1-i}}{(k-1-i)!}+O(n^{j-2})
\end{equation}
as $n\to\infty$.
\end{thm}
It is worth noting that $\sigma_0=1$, and $\sigma_i=0$ if $i$ is odd. Furthermore, if we set $s_i=a_1^i+a_2^i+\cdots+a_k^i$, then
\begin{align*}
\sigma_2=-\frac{s_2}{24}\text{,}\hspace{0.2cm}\sigma_4=\frac{5s_2^2+2s_4}{5760}\text{,}\hspace{0.2cm}\sigma_6=-\frac{35s_2^3+42s_2s_4+16s_6}{2903040}.
\end{align*}

\begin{re}\label{Remark1}
{\rm It is a well-known fact that $p_\mathcal{A}(n,k)$ is a so-called $quasi$-$polynomial$, that is, an expression of the form
\begin{align*}
    p_\mathcal{A}(n,k)=c_{k-1}(r)n^{k-1}+c_{k-2}(r)n^{k-2}+\cdots+c_0(r),
\end{align*}
where each $c_j(r)$ depends on the residue class $r$ of $n\pmod*{\lcm(a_1,a_2,\ldots,a_k)}$ for $0\leqslant j \leqslant k-1$ (for more information about quasi-polynomials, see Stanley \cite[Section~4.4]{Stanley}). The first proof of this result is probably due to Bell \cite{Bell}. Hence, we see that Theorem \ref{2.4} essentially says that (under the assumptions from the statement) the coefficients of the first $k-j+1$ highest degrees of $p_\mathcal{A}(n,k)$ are independent of a residue class of $n\pmod*{\lcm(a_1,a_2,\ldots,a_k)}$.
}
\end{re}

It would be very convenient to find functions, say $f$ and $g$, such that $f(n)<p_\mathcal{A}(n,k)<g(n)$ for each positive integer $n$ and $f(a)f(b)>g(a+b)$ for all $a,b\geqslant N$ and some integer $N$. Since, from these inequalities one can immediately deduce the Bessenrodt-Ono type of inequality for $p_\mathcal{A}(n,k)$ (as well as the $\log$-concavity of the sequence $\left(p_\mathcal{A}(n,k)\right)_{n\geqslant N}$, if we replace the condition $f(a)f(b)>g(a+b)$ by $f^2(a)>g(a+1)g(a-1)$.) The existence of such functions follows from Theorem \ref{2.3} and Theorem \ref{2.4}; and, as we suspect, the main part of the sequel is devoted to discovering them and determining a value of $N$ as small as possible.

\section{The Bessenrodt-Ono type of inequality for $p_\mathcal{A}(n,k)$}

First of all, we focus on the Bessenrodt-Ono type of inequality for $p_\mathcal{A}(n,k)$; mainly because it is the more straightforward task than the $\log$-concavity problem, and provides some general methods which might be also successfully applied elsewhere. Therefore, let $\mathcal{A}=\left(a_i\right)_{i=1}^\infty$ be an increasing sequence of positive integers. Our aim is to resolve under what conditions on positive integers $a,b$ and $k$ and the set $\{a_1,a_2,\ldots,a_k\}$,  the inequality
\begin{align*}
p_\mathcal{A}(a,k)p_\mathcal{A}(b,k)>p_\mathcal{A}(a+b,k)
\end{align*}
holds. In order to do that, we plan to bound the restricted partition function from above and below --- as we just mentioned at the end of the previous section.

At first, let us assume that $k=2$.  In that case we might apply an explicit formula for $p_\mathcal{A}(n,2)$, which was firstly obtained by Popoviciu \cite[Lemma 11]{Pop}, and derive accurate estimates.

\begin{lm}\label{pr3}
Let $\mathcal{A}=\left(a_i\right)_{i=1}^\infty$ be an increasing sequence of positive integers with $\gcd(a_1,a_2)=1$. Define $a_1'(n)$ and $a_2'(n)$ by $a_1'(n)a_1\equiv-n\pmod*{a_2}$ with $1\leqslant a_1'(n)\leqslant a_2$ and $a_2'(n)a_2\equiv-n\pmod*{a_1}$ with $1\leqslant a_2'(n)\leqslant a_1$, respectively. Then for all $n\geqslant1$, we have
\begin{equation*}
p_\mathcal{A}(n,2)=\frac{n+a_1a_1'(n)+a_2a_2'(n)}{a_1a_2}-1.
\end{equation*}
\end{lm}

Lemma \ref{pr3} points out both the lower and the upper bounds for $p_\mathcal{A}(n,2)$, and one can also use it in order to obtain the Bessenrodt-Ono type of inequality, as we see below. 

\begin{cor}\label{cor2}
Let $\mathcal{A}=\left(a_i\right)_{i=1}^\infty$ be an increasing sequence of positive integers such that $\gcd(a_1,a_2)=1$. For all integers $a,b>4a_1a_2$, we have
\begin{align}\label{pr21}
p_\mathcal{A}(a,2)p_\mathcal{A}(b,2)>p_\mathcal{A}(a+b,2).
\end{align}
\end{cor}
\begin{proof}
By Lemma \ref{pr3}, it is clear that
\begin{align*}
\frac{n}{a_1a_2}-1< p_\mathcal{A}(n,2)\leqslant\frac{n}{a_1a_2}+1,
\end{align*}
for each $n\in\mathbb{N}_+$.  Therefore, in particular, if the second inequality in 
\begin{align*}
p_\mathcal{A}(a,2)p_\mathcal{A}(b,2)>\left(\frac{a}{a_1a_2}-1\right)\left(\frac{b}{a_1a_2}-1\right)>\frac{a+b}{a_1a_2}+1\geqslant p_\mathcal{A}(a+b,2)
\end{align*}
holds, then (\ref{pr21}) is valid. But, it is straightforward to observe that the condition is satisfied for all $a,b>4a_1a_2$, as desired.
\end{proof}

\begin{lm}\label{lm1}
Let $\mathcal{A}=\left(a_i\right)_{i=1}^\infty$ be an increasing sequence of positive integers such that $\gcd(a_1,a_2)=1$. If $k\geqslant3$, then the inequalities
\begin{equation}\label{In1}
\frac{n^{k-1}}{(k-1)!\prod_{i=1}^ka_i}-n^{k-2}<p_\mathcal{A}(n,k)<\frac{n^{k-1}}{(k-1)!\prod_{i=1}^ka_i}+n^{k-2}
\end{equation}
hold for each positive integer $n$.

\end{lm}
\begin{proof}
Let $k\geqslant3$ be a fixed integer. At first, we consider the upper bound for $p_\mathcal{A}(n,k)$. The goal of the proof is to determine a constant $C_k$ such that
\begin{align*}
p_\mathcal{A}(n,k)<\frac{n^{k-1}}{(k-1)!\prod_{i=1}^ka_i}+C_kn^{k-2},
\end{align*}
what is possible by Theorem \ref{2.3}. The recurrence formula (\ref{Rf2}) and the induction hypothesis for $k-1$ assert that
\begin{align*}
p_\mathcal{A}(n,k)&=\sum_{i=0}^{\floor{\frac{n}{a_k}}} p_\mathcal{A}(n-ia_k,k-1)\\
&<\frac{1}{(k-2)!\prod_{i=1}^{k-1}a_i}\sum_{i=0}^{\floor{\frac{n}{a_k}}}(n-ia_k)^{k-2}+C_{k-1}\sum_{i=0}^{\floor{\frac{n}{a_k}}}(n-ia_k)^{k-3}\\
&<\frac{1}{(k-2)!\prod_{i=1}^{k-1}a_i}\left(n^{k-2}+\frac{n^{k-1}}{(k-1)a_k}\right)+C_{k-1}\left(n^{k-3}+\frac{n^{k-2}}{(k-2)a_k}\right),
\end{align*} 
where the last inequality is a consequence of the following elementary estimation:
\begin{align*}
\sum_{i=0}^{\floor{\frac{n}{a_k}}}(n-ia_k)^{s} &<n^{s}+\int_0^{\frac{n}{a_k}}(n-xa_k)^{s}dx=n^s+\frac{n^{s+1}}{a_k(s+1)}
\end{align*}
for $s\in\{k-2,k-3\}$. Hence, if the condition
\begin{align*}
\frac{n^{k-1}}{(k-1)!\prod_{i=1}^ka_i}+C_kn^{k-2}\geqslant&\frac{1}{(k-2)!\prod_{i=1}^{k-1}a_i}\left(n^{k-2}+\frac{n^{k-1}}{(k-1)a_k}\right)\\
&+C_{k-1}\left(n^{k-3}+\frac{n^{k-2}}{(k-2)a_k}\right)
\end{align*}
is satisfied, then the constant $C_k$ is as required. However, one can  simplify the above inequality to 
\begin{align*}
C_k\geqslant\frac{1}{(k-2)!\prod_{i=1}^{k-1}a_i}+C_{k-1}\left(\frac{1}{n}+\frac{1}{(k-2)a_k}\right).
\end{align*}
We can demand that $C_{k-1}\geqslant1$. If $n\geqslant6$, then it is enough to take $C_k\geqslant C_{k-1}$ since $\mathcal{A}$ is increasing. Thus, by the proof of Lemma \ref{pr3}, we conclude that $C_k=1$ is as desired. For $n\leqslant5$, it might be easily verified one by one that the upper bound in (\ref{In1}) also holds.

The lower bound can be obtained in a very similar way --- it is actually the much easier task. Hence, we omit the proof, and encourage the reader to verify its correctness on their own. 
\end{proof}

Now, we are ready to prove the main theorem of this section.

\begin{thm}\label{thmbo}
Under the assumptions of Lemma \ref{lm1}, we have
\begin{equation}\label{BO}
p_\mathcal{A}(a,k)p_\mathcal{A}(b,k)>p_\mathcal{A}(a+b,k)
\end{equation}
for all $a,b\geqslant2(k-1)!\prod_{i=1}^ka_i+2$.
\end{thm}
\begin{proof}
Let $k\geqslant3$ be fixed. We want to check under what assumptions on positive integers $a$ and $b$ the condition
\begin{align*}
p_\mathcal{A}(a,k)p_\mathcal{A}(b,k)>p_\mathcal{A}(a+b,k)
\end{align*}
holds. Lemma \ref{lm1} asserts that
\begin{align*}
p_\mathcal{A}(a,k)p_\mathcal{A}(b,k)>\left(\frac{a^{k-1}}{(k-1)!\prod_{i=1}^ka_i}-a^{k-2}\right)\cdot\left(\frac{b^{k-1}}{(k-1)!\prod_{i=1}^ka_i}-b^{k-2}\right)
\end{align*}
and
\begin{align*}
\frac{(a+b)^{k-1}}{(k-1)!\prod_{i=1}^ka_i}+(a+b)^{k-2}>p_\mathcal{A}(a+b,k).
\end{align*} 
Let us put $D_k:=(k-1)!\prod_{i=1}^ka_i$. Therefore, in particular, if the inequality
\begin{align*}
\left(\frac{a^{k-1}}{D_k}-a^{k-2}\right)\cdot\left(\frac{b^{k-1}}{D_k}-b^{k-2}\right)>\frac{(a+b)^{k-1}}{D_k}+(a+b)^{k-2}
\end{align*}
is true, then we are done. This expression might be further simplified to
\begin{align*}
\frac{a^{k-2}b^{k-2}}{D_k}\left(a-D_k\right)\left(b-D_k\right)>(a+b)^{k-2}\left(a+b+D_k\right).
\end{align*}
Hence, it is enough to verify when both of the inequalities
\begin{align}\label{3.2}
\frac{a^{k-2}b^{k-2}}{D_k}>(a+b)^{k-2}
\end{align}
and
\begin{align}\label{3.3}
\left(a-D_k\right)\left(b-D_k\right)>a+b+D_k
\end{align}
hold.
The second one maintains that $a,b>D_k$. However, let us first deal with (\ref{3.2}). Without loss of generality, we may assume that $a\geqslant b$ (otherwise we simple replace $a$ and $b$ with each other). Now, it is clear that
\begin{align*}
\frac{a^{k-2}b^{k-2}}{D_k}>(2a)^{k-2}
\end{align*}
implies (\ref{3.2}). Thus, we see that for $k\geqslant3$ all $a,b>2D_k$ satisfy (\ref{3.2}).

On the other hand, in the second case, we can easily notice that (\ref{3.3}) follows from the inequality
\begin{align*}
ab-2a(D_k+1)+D_k(D_k-1)>0.
\end{align*}
Since $D_k>1$, it is enough to take any $a,b\geqslant2(D_k+1)$. 

Summing up, the Bessenrodt-Ono inequality (\ref{BO}) holds for every $a,b\geqslant2(D_k+1)$ and $k\geqslant3$, as desired. Finally, the proof is complete.

\end{proof}
\begin{re}
{\rm 
There is a more general version of the foregoing theorem in Sec. 5 (Theorem \ref{Thm5.1}), however there are significantly stronger assumptions on the numbers $a$ and $b$.
}
\end{re}
At the moment, one might ask: what will there happen if we allow $k$ to be equal $1$? Actually, Proposition \ref{Rf} automatically points out that in this case, there is no constant $N$ such that 
\begin{align*}
p_\mathcal{A}(a,1)p_\mathcal{A}(b,1)>p_\mathcal{A}(a+b,1)
\end{align*}
for all $a,b>N$.

Further, it is worth noting that the requirements for $a$ and $b$ in Theorem \ref{thmbo} are not optimal. Mainly because, we do not have any additional assumptions on the sequence $\mathcal{A}=\left(a_i\right)_{i=1}^\infty$ other than the monotonicity and $\gcd(a_1,a_2)=1$.

We finish this section with an elementary application of the Bessenrodt-Ono type inequality  (\ref{BO}).

\begin{ex}\label{Example1}
Let $\mathcal{A}=\left(2^{i-1}\right)_{i=1}^\infty$ be the sequence of consecutive powers of two. Theorem \ref{thmbo} implies that for every $k\geqslant3$ and all $a,b>2^{k(k-1)/2+1}(k-1)!+1$, we have
\begin{align*}
p_\mathcal{A}(a,k)p_\mathcal{A}(b,k)>p_\mathcal{A}(a+b,k).
\end{align*}
In the case of $k=2$, Corollary \ref{cor2} asserts that the above inequality is valid for all $a,b>8$. But, it is well-known that $p_\mathcal{A}(n,2)=\floor{\frac{n}{2}}+1$ for each $n\in\mathbb{N}$. Hence, one might find out that
\begin{align*}
p_\mathcal{A}(a,2)p_\mathcal{A}(b,2)>p_\mathcal{A}(a+b,2)
\end{align*}
holds for every $a,b>1$ except for $(a,b)=(3,3)$. In other words, Corollary \ref{cor2} is not optimal as well.
\end{ex}

\section{Log-concavity of $p_\mathcal{A}(n,k)$}

The next part of the paper is devoted to the $\log$-concavity of the restricted partition function $p_\mathcal{A}(n,k)$. Namely, we investigate under what assumptions on an increasing sequence of positive integers $\mathcal{A}$ and  constants $k\in\mathbb{N}_+$ and $N\in\mathbb{R}$, the inequality
\begin{align*}
p_\mathcal{A}^2(n,k)>p_\mathcal{A}(n+1,k)p_\mathcal{A}(n-1,k)
\end{align*} 
holds for each $n>N$.

It is worth underlying that it is a more complex task than the foregoing one. However, with the actual state of knowledge, we can suppose that the approach from the previous section may be also effective. Nonetheless, there appears a subtle issue in the reasoning. More precisely, if we apply the idea from the proof of Theorem \ref{thmbo}, then in some step we obtain the inequality of the form
\begin{align*}
(n-D_k)^2>(n+1+D_k)(n-1+D_k)
\end{align*}
with $D_k$ as before, which is never satisfied for large values of $n$. Furthermore, even if we use Theorem \ref{2.4} in order to determine a bit more accurate estimates for $p_\mathcal{A}(n,k)$ --- that is a number $E_k$ such that 
\begin{align*}
\alpha n^{k-1}+\beta n^{k-2}-E_kn^{k-3}<p_\mathcal{A}(n,k)<\alpha n^{k-1}+\beta n^{k-2}+E_kn^{k-3},
\end{align*} 
where $\alpha=\frac{1}{(k-1)!\prod_{i=1}^{k}a_i}$ and $\beta=\frac{\sigma}{2(k-2)!\prod_{i=1}^{k}a_i}$ with $\sigma=a_1+a_2+\cdots+a_k$, holds for every positive integer $n$ --- then we encounter a similar problem to that one described a few lines above. We might nearly expect that our approach is ineffective in the $\log$-concavity problem. Surprisingly, it turns out that, if we find out even more precise lower and upper bounds for $p_\mathcal{A}(n,k)$, then the obstacle, which plagued us above, suddenly vanishes. Therefore, the first part of this section deals with deriving a constant $E_k$ such that  
\begin{align*}
(\alpha n^{2}+\beta n+\gamma)n^{k-3}-E_kn^{k-4}<p_\mathcal{A}(n,k)<(\alpha n^{2}+\beta n+\gamma) n^{k-3}+E_kn^{k-4},
\end{align*}  
where $\alpha$ and $\beta$ are as before, and $\gamma=\frac{3\sigma^2-s_2}{2^3\cdot3(k-3)!\prod_{i=1}^{k}a_i}$ with $\sigma=a_1+a_2+\cdots+a_k$ and $s_2=a_1^2+a_2^2+\cdots+a_k^2$ --- which follows once again from Theorem \ref{2.4}. In order to derive $E_{k}$ from $E_{k-1}$ we will use both Theorem \ref{2.4} and a well-known Euler-Maclaurin formula, which may be found, for instance, in Apostol's paper \cite{Ap1}.

\begin{thm}\label{EM}
If $u,v\in\mathbb{N}$ are such that $u<v$, and $f$ is a $p$ times continuously differentiable function in the interval $[u,v]$, then 
\begin{equation*}
    \sum_{i=u}^vf(i)=\int_u^vf(x)dx+\frac{f(v)+f(u)}{2}+\sum_{j=1}^{\floor{\frac{p}{2}}}\frac{B_{2j}}{(2j)!}\left(f^{(2j-1)}(v)-f^{(2j-1)}(u)\right)+R_p,
\end{equation*}
where $B_s$ is the $s$-th Bernoulli number (with $B_1=\frac{1}{2}$) and $R_p$ is an error term which depends on $u, v, p,$ and $f$. Moreover,
\begin{align*}
    |R_p|\leqslant\frac{2\zeta(p)}{(2\pi)^p}\int_u^v|f^{(p)}(x)|dx,
\end{align*}
where $\zeta$ denotes the Riemann zeta function.
\end{thm}

However, we also need to deal with $E_4$ separately, and in order to do that we apply a formula for $p_\mathcal{A}(n,k)$ obtained by Cimpoeaş and Nicolae \cite{CN}.
\begin{thm}\label{CNT}
For an increasing sequence of positive integers $\mathcal{A}=\left(a_i\right)_{i=1}^\infty$ and $k\geqslant1$, we have
\begin{align*}
    p_\mathcal{A}(n,k)=\frac{1}{(k-1)!}&\sum_{m=0}^{k-1}\sum_{\substack{ 0\leqslant j_1\leqslant\frac{D}{a_1}-1,\ldots,0\leqslant j_k\leqslant\frac{D}{a_k}-1\\
    a_1j_1+\cdots+a_kj_k\equiv n\pmod*{D}}}\sum_{i=m}^{k-1}\stirlingone{k}{i}(-1)^{i-m}\binom{i}{m}\\
    &\times D^{-i}(a_1j_1+\cdots+a_kj_k)^{i-m}n^m,
\end{align*}
where $D=\lcm(a_1,a_2,\ldots,a_k)$ and $\stirlingone{k}{i}$ is the unsigned Stirling number of the first kind.
\end{thm}
\begin{lm}\label{k=4}
Let $\mathcal{A}=(a_i)_{i=1}^\infty$ be an increasing sequence of positive integers such that $\gcd(a_i,a_j)=1$ for $1\leqslant i<j\leqslant4$. Let $\alpha=\frac{1}{3!\prod_{i=1}^{4}a_i}$, $\beta=\frac{\sigma}{4\prod_{i=1}^{4}a_i}$ and  $\gamma=\frac{3\sigma^2-s_2}{24\prod_{i=1}^{4}a_i}$ with $\sigma=a_1+a_2+a_3+a_4$ and $s_2=a_1^2+a_2^2+a_3^2+a_4^2$. We have
\begin{align*}
(\alpha n^{2}+\beta n+\gamma)n-16(a_1a_2a_3a_4)^3<p_\mathcal{A}(n,4)<(\alpha n^{2}+\beta n+\gamma) n+16(a_1a_2a_3a_4)^3.
\end{align*}
\end{lm}

\begin{proof}
At the beginning, we notice that Theorem \ref{2.4} points out that $\alpha$, $\beta$ and $\gamma$ are as required. Hence, we just want to bound from above and below the constant term of the quasi-polynomial $p_A(n,4)=\alpha n^{3}+\beta n^2+\gamma n+c_0(n)$, where $c_0(n)$ depends on a residue class of $n\pmod*{(a_1a_2a_3a_4)}$. At first, let us deal with the lower estimate for $c_0(n)$. Theorem \ref{CNT} asserts that
\begin{align*}
    c_0(n)&=\frac{1}{3!}\sum_{\substack{ 0\leqslant j_1\leqslant a_2a_3a_4-1,\ldots,0\leqslant j_4\leqslant a_1a_2a_3-1\\
    a_1j_1+\cdots+a_4j_4\equiv n\pmod*{a_1a_2a_3a_4}}}\sum_{i=0}^{3}\stirlingone{4}{i}(-1)^i\left(\frac{a_1j_1+\cdots+a_4j_4}{a_1a_2a_3a_4}\right)^i\\
    &>-\frac{1}{3!}\sum_{j_1=0}^{a_2a_3a_4-1}\sum_{j_2=0}^{a_1a_3a_4-1}\sum_{j_3=0}^{a_1a_2a_4-1}\sum_{j_4=0}^{a_1a_2a_3-1}\sum_{i\in\{1,3\}}\stirlingone{4}{i}\left(\frac{a_1j_1+\cdots+a_4j_4}{a_1a_2a_3a_4}\right)^i\\
    &>-2(a_1a_2a_3a_4)^3-\frac{1}{(a_1a_2a_3a_4)^3}\bigg{(}10(a_1a_2a_3a_4)^6+(a_1a_2a_3a_4)^4\sum_{i=1}^4a_i^2\\
    &\phantom{>}+3(a_1a_2a_3a_4)^4\sum_{1\leqslant i<j\leqslant4}a_ia_j\bigg{)}>-16(a_1a_2a_3a_4)^3,
\end{align*}
where the penultimate inequality is the consequence of elementary but tiresome computations; the last one, on the other hand, follows from the fact that $a_i\geqslant i$ for each $i\in\{1,2,3,4\}$. The proof of the lower bound is complete. In order to obtain the upper bound, it is enough to repeat our reasoning. Actually, it is the easier task, hence details are left for the reader to verify on their own.

\end{proof}

Before we generalize the above result for each $k\geqslant4$, let us say something about the $\log$-concavity of $p_\mathcal{A}(n,4)$ in that special case.

\begin{pr}\label{pr4}
Under the assumptions of Lemma \ref{k=4}, the inequality
\begin{align*}
p_\mathcal{A}^2(n,4)>\left(1+\frac{1}{n^2}\right)p_\mathcal{A}(n+1,4)p_\mathcal{A}(n-1,4)
\end{align*}
holds for every $n\geqslant288(a_1a_2a_3a_4)^4$. In particular, the sequence $\left(p_\mathcal{A}(n,4)\right)_{n=1}^\infty$  is log-concave for all $n\geqslant 192(a_1a_2a_3a_4)^4$.
\end{pr}

\begin{proof}
Lemma \ref{k=4} implies that
\begin{align*}
    p_\mathcal{A}^2(n,4)>(\alpha n^3+\beta n^2+\gamma n-16(a_1a_2a_3a_4)^3)^2
\end{align*}
and
\begin{align*}
    p_\mathcal{A}(n+1,4)p_\mathcal{A}(n-1,4)&<(\alpha(n+1)^{3}+\beta(n+1)^2+\gamma(n+1)+16(a_1a_2a_3a_4)^3)\\
    &\times(\alpha(n-1)^{3}+\beta(n-1)^2+\gamma(n-1)+16(a_1a_2a_3a_4)^3).
\end{align*}
Analogously to the proof of Theorem \ref{thmbo}, we consider the inequality:
\begin{align*}
    &(\alpha n^3+\beta n^2+\gamma n-16(a_1a_2a_3a_4)^3)\times(\alpha n^3+\beta n^2+\gamma n-16(a_1a_2a_3a_4)^3)\\
    &>\left(1+\frac{1}{n^2}\right)(\alpha(n+1)^{3}+\beta(n+1)^2+\gamma(n+1)+16(a_1a_2a_3a_4)^3)\\
    &\phantom{>\left(1+\frac{1}{n^2}\right)}\times(\alpha(n-1)^{3}+\beta(n-1)^2+\gamma(n-1)+16(a_1a_2a_3a_4)^3).
\end{align*}
It might be further reduced as follows
\begin{align*}
    2n^6&+(2a-4c)n^5+(a^2-2b-4ac)n^4+(2a-8c-4bc)n^3+(a^2+2b-4ac-2)n^2\\
    &+(-2a+2ab-6c-2bc)n+1-a^2+2b+b^2-2ac-c^2>0,
\end{align*}
where $a=\frac{3\sigma}{2}$, $b=\frac{3\sigma^2-s_2}{4}$ and $c=96(a_1a_2a_3a_4)^4$. Now, let us denote the sum on the left hand side by $g(n)$. It turns out that the leading coefficient of $g^{(4)}(n)$ (the fourth derivative of $g$) is positive, and its both real roots are given by:
\begin{align*}
    n_1&=\frac{1}{30}(-5a+10c-\sqrt{5}\sqrt{-a^2+12b+4ac+20c^2}),\\
    n_2&=\frac{1}{30}(-5a+10c+\sqrt{5}\sqrt{-a^2+12b+4ac+20c^2}).
\end{align*}
One can easily see that $n_2<c$ --- in particular, it means that $g^{(3)}(n)$ is increasing for $n\geqslant c$. At the moment, our reasoning is based on the following elementary observation: if the value of $g^{(i)}(2c)$ is positive, then $g^{(i)}(n)>0$ and $g^{(i-1)}(n)$ is increasing for all $n\geqslant2c$. One may verify that, indeed, $g^{(i)}(2c)>0$ for each $i\in\{1,2,3,4\}$. However, it might happen that $g(2c)<0$ --- nevertheless, $g(3c)>0$; and the first part of the statement follows. The proof of the $\log$-concavity is very similar, actually, it is enough to consider the second derivative of the corresponding polynomial $\widetilde{g}$ instead of the fourth one. In contrast to the above proof, we also have $\widetilde{g}(2c)>0$. The details are left for the reader. 
\end{proof}
\begin{re}\label{remarkk=4}
{\rm 
It is worth underlying that the term $(1+1/n^2)$ appearing in the above proposition is optimal, in the sense that we can not replace $2$ by any smaller exponent. In order to observe this fact, we need to remind ourselves that $p_\mathcal{A}(n,4)$ is a quasi-polynomial (Remark \ref{Remark1}). Since we have some additional assumptions on $a_1,a_2,a_3$ and $a_4$ in Proposition \ref{pr4}, it is clear that for every $n\pmod*{a_1a_2a_3a_4}$ the restricted partition function takes the form:
\begin{align*}
    p_\mathcal{A}(n,4)=\alpha n^3+\beta n^2+\gamma n+c_0(n),
\end{align*}
where $c_0(n)$ depends on the residue class of $n\pmod*{a_1a_2a_3a_4}$, and $\alpha$, $\beta$ and $\gamma$ are as in Lemma \ref{k=4}.

Now, it suffices to observe that the leading coefficient of  
\begin{align*}
    p_\mathcal{A}^2(n,4)n^s-(n^s+1)p_\mathcal{A}(n+1,4)p_\mathcal{A}(n-1,4)
\end{align*}
is independent of the terms $c_0(j)$ for $j=n-1,n,n+1$; and is negative if and only if $s<2$. Thus, $s=2$ is optimal, as desired.

However, some elementary computations show that the term $(1+1/n^2)$ might be replaced by  $(1+1/(un^2))$ for every $u>1/3$. Since in that case, the leading coefficient of the obtained polynomial $g$ (in the proof) remains positive. Moreover, one can also show that
\begin{align*}
p_\mathcal{A}^2(n,4)>\left(1+\frac{1}{n^2}\right)^2p_\mathcal{A}(n+1,4)p_\mathcal{A}(n-1,4)
\end{align*}
holds for sufficiently large values of $n$.
}
\end{re}

\begin{lm}\label{k>4}
Let $\mathcal{A}=\left(a_i\right)_{i=1}^\infty$ be an increasing sequence of positive integers such that $\gcd(a_i,a_j)=1$ for $1\leqslant i<j\leqslant4$. Let $\alpha=\frac{1}{(k-1)!\prod_{i=1}^{k}a_i}$, $\beta=\frac{\sigma}{2(k-2)!\prod_{i=1}^{k}a_i}$ and  $\gamma=\frac{3\sigma^2-s_2}{24(k-3)!\prod_{i=1}^{k}a_i}$ with $\sigma=a_1+a_2+\cdots+a_k$ and $s_2=a_1^2+a_2^2+\cdots+a_k^2$.
If $k\geqslant4$ and $n\geqslant a_k$, then
\begin{align*}
(\alpha n^{2}+\beta n+\gamma)n^{k-3}-E_kn^{k-4}<p_\mathcal{A}(n,k)<(\alpha n^{2}+\beta n+\gamma) n^{k-3}+E_kn^{k-4},
\end{align*}  
where $E_k=\frac{k^2(a_1a_2a_3)^3a_k^{k+3}}{\prod_{i=1}^ka_i}\cdot e^\frac{1}{a_k}$.
\end{lm}
\begin{proof}
As in the proof of Lemma \ref{lm1} , the reasoning in both cases is very similar. Therefore, we deal only with the lower bound for $p_\mathcal{A}(n,k)$, and encourage the reader to verify the correctness of the statement in the remaining case.

If $k=4$, then it is clear by Lemma \ref{k=4}. Let $k\geqslant5$. By Theorem \ref{2.4}, our aim is to determine the constant $E_k>0$ such that the inequality
$$p_\mathcal{A}(n,k)>\frac{n^{k-1}}{(k-1)!\prod_{i=1}^ka_i}+\frac{\sigma n^{k-2}}{2(k-2)!\prod_{i=1}^ka_i}+\frac{(3\sigma^2-s_2)n^{k-3}}{2^3\cdot3(k-3)!\prod_{i=1}^ka_i}-E_kn^{k-4}$$
holds for each $n\geqslant1$, where $\sigma$ and $s_2$ are as in the statement. Corollary \ref{Rf2} asserts that
\begin{align*}
    p_\mathcal{A}(n,k)=\sum_{j=0}^{\floor{\frac{n}{a_k}}}p_\mathcal{A}(n-ja_k,k-1).
\end{align*}
Thus, it is enough to find out when the inequality
\begin{align*}
    &\frac{1}{(k-2)!\prod_{i=1}^{k-1}a_i}\sum_{j=0}^{\floor{\frac{n}{a_k}}} (n-ja_k)^{k-2}+\frac{\sigma-a_k}{2(k-3)!\prod_{i=1}^{k-1}a_i}\sum_{j=0}^{\floor{\frac{n}{a_k}}} (n-ja_k)^{k-3}\\
    &+\frac{3(\sigma-a_k)^2-s_2+a_k^2}{2^3\cdot3(k-4)!\prod_{i=1}^{k-1}a_i}\sum_{j=0}^{\floor{\frac{n}{a_k}}} (n-ja_k)^{k-4}-E_{k-1}\sum_{j=0}^{\floor{\frac{n}{a_k}}} (n-ja_k)^{k-5}\\
    &>\frac{n^{k-1}}{(k-1)!\prod_{i=1}^ka_i}+\frac{\sigma n^{k-2}}{2(k-2)!\prod_{i=1}^ka_i}+\frac{(3\sigma^2-s_2)n^{k-3}}{2^3\cdot3(k-3)!\prod_{i=1}^ka_i}-E_kn^{k-4}
\end{align*}
is satisfied. We just estimate each summand on the left hand side separately. In order to do that, we apply Theorem \ref{EM} with $u=0$, $v=\floor{\frac{n}{a_k}}$, $p=s-(k-5)$ and $f(x)=(n-xa_k)^s$ for $s\in\{k-2,k-3\}$, namely,
\begin{align*}
    \sum_{j=0}^{\floor{\frac{n}{a_k}}}(n-ja_k)^s=&\int_0^{\floor{\frac{n}{a_k}}}(n-xa_k)^sdx+\frac{(n\pmod*{a_k})^{s}+n^{s}}{2}\\
    &+\frac{a_ks}{6\cdot2!}\left(n^{s-1}-(n-\pmod*{a_k})^{s-1}\right)+R_{p}.
\end{align*}
Since either $p=2$ or $p=3$, we have
\begin{align*}
    |R_{p}|\leqslant\frac{2\zeta(p)}{(2\pi)^p}\int_0^{\floor{\frac{n}{a_k}}}|f^{(p)}(x)|dx\leqslant\frac{2\zeta(2)}{(2\pi)^2}\int_0^{\floor{\frac{n}{a_k}}}|f^{(p)}(x)|dx.
\end{align*}
After some elementary computations, one can deduce that
\begin{align*}
    \sum_{j=0}^{\floor{\frac{n}{a_k}}}(n-ja_k)^{k-2}\geqslant\frac{n^{k-1}-a_k^{k-1}}{a_k(k-1)}+\frac{n^{k-2}}{2}&+\frac{a_k(k-2)(n^{k-3}-a_k^{k-3})}{12}\\
    &-\frac{a_k^2(k-2)(k-3)n^{k-4}}{12}
\end{align*}
as well as
\begin{align*}
    \sum_{j=0}^{\floor{\frac{n}{a_k}}}(n-ja_k)^{k-3}\geqslant\frac{n^{k-2}-a_k^{k-2}}{a_k(k-2)}+\frac{n^{k-3}}{2}-\frac{a_k^{k-3}(k-3)}{12}.
\end{align*}
Further, we have the following elementary estimation:
\begin{align*}
    \sum_{j=0}^{\floor{\frac{n}{a_k}}}(n-ja_k)^{k-4}>\int_0^{\floor{\frac{n}{a_k}}}(n-xa_k)^{k-4}dx\geqslant\frac{n^{k-3}-a_k^{k-3}}{a_k(k-3)}.
\end{align*}
Finally, the sum next to $E_{k-1}$ can be simply bounded from above by $n^{k-4}$. Hence, it is clear that if 
\begin{align*}
    &\frac{1}{(k-2)!\prod_{i=1}^{k-1}a_i}\Bigg{[}\frac{n^{k-1}-a_k^{k-1}}{a_k(k-1)}+\frac{n^{k-2}}{2}+\frac{a_k(k-2)(n^{k-3}-a_k^{k-3})}{12}\\
    &-\frac{a_k^2(k-2)(k-3)n^{k-4}}{12}\Bigg{]}+\frac{\sigma-a_k}{2(k-3)!\prod_{i=1}^{k-1}a_i}\Bigg{[}\frac{n^{k-2}-a_k^{k-2}}{a_k(k-2)}+\frac{n^{k-3}}{2}\\
    &-\frac{a_k^{k-3}(k-3)}{12}\Bigg{]}+\frac{3(\sigma-a_k)^2-s_2+a_k^2}{2^3\cdot3(k-4)!\prod_{i=1}^{k-1}a_i}\cdot\frac{n^{k-3}-a_k^{k-3}}{a_k(k-3)}-E_{k-1}n^{k-4}\\
    &\geqslant\frac{n^{k-1}}{(k-1)!\prod_{i=1}^ka_i}+\frac{\sigma n^{k-2}}{2(k-2)!\prod_{i=1}^ka_i}+\frac{(3\sigma^2-s_2)n^{k-3}}{2^3\cdot3(k-3)!\prod_{i=1}^ka_i}-E_kn^{k-4}
\end{align*}
holds, then $E_k$ is as desired. The assumptions that $n\geqslant a_k$ and $\mathcal{A}$ is increasing together with some basic reduction maintain that the inequality
\begin{align*}
    E_k\geqslant E_{k-1}&+\frac{a_k^2}{(k-1)!\prod_{i=1}^{k-1}a_i}+\frac{a_k^2}{12(k-3)!\prod_{i=1}^{k-1}a_i}+\frac{a_k^2}{12(k-4)!\prod_{i=1}^{k-1}a_i}\\
    &+\frac{ka_k^2}{2(k-2)!\prod_{i=1}^{k-1}a_i}+\frac{ka_k^2}{24(k-4)!\prod_{i=1}^{k-1}a_i}+\frac{3k^2a_k^2}{24(k-3)!\prod_{i=1}^{k-1}a_i}
\end{align*}
implies the foregoing one. Since $k>4$, we may simplify the expression on the right hand side as follows:
\begin{align*}
    &\frac{a_k^2}{(k-1)!\prod_{i=1}^{k-1}a_i}+\frac{a_k^2}{12(k-3)!\prod_{i=1}^{k-1}a_i}+\frac{a_k^2}{12(k-4)!\prod_{i=1}^{k-1}a_i}+\frac{ka_k^2}{2(k-2)!\prod_{i=1}^{k-1}a_i}\\
    &+\frac{ka_k^2}{24(k-4)!\prod_{i=1}^{k-1}a_i}+\frac{k^2a_k^2}{8(k-3)!\prod_{i=1}^{k-1}a_i}<\frac{6\cdot k^2a_k^2}{(k-4)!\prod_{i=1}^{k-1}a_i}<\frac{k^2(a_1a_2a_3a_k)^3}{(k-4)!\prod_{i=1}^{k-4}a_i}.
\end{align*}
Thus, it is enough to take $E_k$ such that
\begin{align*}
    E_k&\geqslant E_{k-1}+\frac{k^2(a_1a_2a_3a_k)^3}{(k-4)!\prod_{i=1}^{k-4}a_i}\tag{$\star$}\\&=E_{k-2}+\frac{(k-1)^2(a_1a_2a_3a_{k-1})^3}{(k-5)!\prod_{i=1}^{k-5}a_i}+\frac{k^2(a_1a_2a_3a_k)^3}{(k-4)!\prod_{i=1}^{k-4}a_i}\\
    &\vdots\\
    &=\frac{4^2(a_1a_2a_3a_4)^3}{0!\prod_{i=1}^{0}a_i}+\frac{5^2(a_1a_2a_3a_5)^3}{1!\prod_{i=1}^{1}a_i}+\cdots+\frac{k^2(a_1a_2a_3a_k)^3}{(k-4)!\prod_{i=1}^{k-4}a_i},
\end{align*}
with empty product defined to be $1$. The last sum can be finally bounded from above by
\begin{align*}
    \frac{k^2(a_1a_2a_3)^3}{\prod_{i=1}^ka_i}\left(\frac{a_k^{k+3}}{0!}+\frac{a_k^{k+2}}{1!}+\cdots+\frac{a_k^{7}}{(k-4)!}\right)&<\frac{k^2(a_1a_2a_3)^3a_k^{k+3}}{\prod_{i=1}^ka_i}\sum_{j=0}^\infty\frac{1}{j!a_k^j}\\&=\frac{k^2(a_1a_2a_3)^3a_k^{k+3}}{\prod_{i=1}^ka_i}\cdot e^\frac{1}{a_k}.
\end{align*}
Hence, we might just set $E_k:=\frac{k^2(a_1a_2a_3)^3a_k^{k+3}}{\prod_{i=1}^ka_i}\cdot e^\frac{1}{a_k}$, as desired.
\end{proof}

Naturally, the value of $E_k$ in Lemma \ref{k>4} is not optimal, but as we see in the foregoing proof, it might be quite messy to derive a bit better estimate for that coefficient. Eventually, it is time to present the main result of this section.

\begin{thm}\label{log1}
Under the assumptions of Lemma \ref{k>4} with $k>4$, the inequality
\begin{align}\label{l1}
    p_\mathcal{A}^2(n,k)>\left(1+\frac{1}{n^2}\right)p_\mathcal{A}(n+1,k)p_\mathcal{A}(n-1,k)
\end{align}
holds for each $n>2k^2(k-1)!(a_1a_2a_3)^3a_k^{k+3}e^\frac{1}{a_k}$, and, in particular, the sequence $\left(p_\mathcal{A}(n,k)\right)_{n=1}^\infty$ is log-concave for all such $n$.
\end{thm}
\begin{proof}

Lemma \ref{k>4} asserts that for all $k\geqslant4$ and $n>a_k$, we have
\begin{align*}
    p_\mathcal{A}^2(n,k)>[(\alpha n^{2}+\beta n+\gamma)n^{k-3}-E_kn^{k-4}]^2
\end{align*}
and
\begin{align*}
    p_\mathcal{A}(n+1,k)p_\mathcal{A}(n-1,k)&<[(\alpha (n+1)^{2}+\beta (n+1)+\gamma)(n+1)^{k-3}+E_k(n+1)^{k-4}]\\
    &\times[(\alpha (n-1)^{2}+\beta (n-1)+\gamma)(n-1)^{k-3}+E_k(n-1)^{k-4}],
\end{align*}
where $\alpha$, $\beta$, $\gamma$ and $E_k$ are as before. Therefore, it is enough to answer when the following inequality holds:
\begin{align*}
    &[(\alpha n^{2}+\beta n+\gamma)n^{k-3}-E_kn^{k-4}]\times[(\alpha n^{2}+\beta n+\gamma)n^{k-3}-E_kn^{k-4}]\\
    &>\left(1+\frac{1}{n^2}\right)[(\alpha (n+1)^{2}+\beta (n+1)+\gamma)(n+1)^{k-3}+E_k(n+1)^{k-4}]\\
    &\phantom{>(1+\frac{1}{n^2}.}\times[(\alpha (n-1)^{2}+\beta (n-1)+\gamma)(n-1)^{k-3}+E_k(n-1)^{k-4}].
\end{align*}
We can simplify the above expression to
\begin{align*}
    \frac{n^4}{n^4-1}\left(\frac{n^2}{n^2-1}\right)^{k-5}(n^3+an^2+bn-c)^2&>[(n+1)^3+a(n+1)^2+b(n+1)+c]\\
    &\times[(n-1)^3+a(n-1)^2+b(n-1)+c],
\end{align*}
where $a=\frac{(k-1)\sigma}{2}$, $b=\frac{(k-1)(k-2)(3\sigma^2-s_2)}{24}$ and $c=k^2(k-1)!(a_1a_2a_3)^3a_k^{k+3}e^\frac{1}{a_k}$. Hence, we might just determine for which values of $n$ the condition
\begin{align*}
   (n^3+an^2+bn-c)^2&>[(n+1)^3+a(n+1)^2+b(n+1)+c]\\
   &\times[(n-1)^3+a(n-1)^2+b(n-1)+c]
\end{align*}
holds. In fact, it is the same problem as the proof of the $\log$-concavity in Proposition \ref{pr4}. Thus, the requirements for the parameter $n$ follow.
\end{proof}

\begin{re}
{\rm 
Analogously to Remark \ref{remarkk=4}, it turns out that the exponent $2$ appearing in the denominator of $(1+1/n^2)$ is optimal. In that case we can de facto replace $(1+1/n^2)$ by $(1+1/(un^2))$, where $u>1/(k-1)$ is arbitrary. Furthermore, one can also show that 
\begin{align*}
    p_\mathcal{A}^2(n,k)>\left(1+\frac{1}{n^2}\right)^{k-2}p_\mathcal{A}(n+1,k)p_\mathcal{A}(n-1,k)
\end{align*}
is satisfied for sufficiently large values of $n$.
}
\end{re}

Even though Theorem \ref{log1} gives us a $\log$-concavity criterion for wide class of integer sequences, we still do not know whether the restricted partition function for the most natural example, namely, the sequence of consecutive positive integers $\mathcal{A}_1:=(1,2,3,4,5,\ldots)$ might be $\log$-concave or not --- let us investigate this issue now.

Clearly, $p_{\mathcal{A}_1}(n,1)$ is not $\log$-concave for all $n\geqslant1$. Furthermore, Example \ref{Example1} asserts that if only $n$ is odd, then $p_{\mathcal{A}_1}^2(n,2)<p_{\mathcal{A}_1}(n+1,2)p_{\mathcal{A}_1}(n-1,2)$. A lot of formulas for certain small values of $k$ are gathered in \cite{GL1, GL2, Gupta, Munagi, Sills}. For instance (see, \cite[Chapter~6]{GA1}), we have 
\begin{align}
    p_{\mathcal{A}_1}(n,3)&=\flce{\frac{(n+3)^2}{12}},\label{4.2}\\
    p_{\mathcal{A}_1}(n,4)&=\flce{(n+5)\left(n^2+n+22+18\floor{\frac{n}{2}}\right)/144},\label{4.3}\\
    \label{p_A(n,5)} p_{\mathcal{A}_1}(n,5)&=\flce{(n+8)\left(n^3+22n^2+44n+248+180\floor{\frac{n}{2}}\right)/2880},
\end{align}
where $\flce{\cdot}$ is the nearest integer function. In fact, the arguments of the function $\flce{\cdot}$ appearing in (\ref{4.2}), (\ref{4.3}) and (\ref{p_A(n,5)}) can not be half-integers. Now, employing the above formulas for $k\in\{3,4\}$, one might check (after some tedious but elementary calculations) that $$p_{\mathcal{A}_1}^2(n,k)<p_{\mathcal{A}_1}(n+1,k)p_{\mathcal{A}_1}(n-1,k)$$ for every $n\equiv-1\pmod{\lcm(1,\ldots,k)}$. Thus, in particular, the sequence \linebreak $\left(p_{\mathcal{A}_1}(n,k)\right)_{n=1}^\infty$ is not $\log$-concave. However, (\ref{p_A(n,5)}) points out that
\begin{align*}
    \frac{n^4}{2880}+\frac{n^3}{96}+\frac{31 n^2}{288}+\frac{41 n}{96}-\frac{11}{180}<p_{\mathcal{A}_1}(n,5)< \frac{n^4}{2880}+\frac{n^3}{96}+\frac{31 n^2}{288}+\frac{11 n}{24}+\frac{107}{90}.
\end{align*}
Next, similar reasoning to that one from the proof of Proposition \ref{k=4} shows that 
\begin{align*}
    p_{\mathcal{A}_1}^2(n,5)>\left(1+\frac{1}{n^2}\right)p_{\mathcal{A}_1}(n+1,5)p_{\mathcal{A}_1}(n-1,5)
\end{align*}
is satisfied for all $n\geqslant81$. For smaller values of $n$, we carry out adequate computations in Wolfram Mathematica \cite{WM}, and obtain the following fact.
\begin{pr}\label{A_1k=5}
For every $n>61$, we have 
\begin{align*}
    p_{\mathcal{A}_1}^2(n,5)>\left(1+\frac{1}{n^2}\right)p_{\mathcal{A}_1}(n+1,5)p_{\mathcal{A}_1}(n-1,5).
\end{align*}
Moreover, the sequence $\left(p_{\mathcal{A}_1}(n,5)\right)_{n=1}^\infty$ is log-concave for all $n>37$.
\end{pr}

Now, let us notice that we might also bound $p_{\mathcal{A}_1}(n,5)$ by
\begin{align*}
    \frac{n^4}{2880}+\frac{n^3}{96}+\frac{31 n^2}{288}-\widetilde{E_5}n<p_{\mathcal{A}_1}(n,5)< \frac{n^4}{2880}+\frac{n^3}{96}+\frac{31 n^2}{288}+\widetilde{E_5}n,
\end{align*}
where $\widetilde{E_5}=675000$ --- that is the value of $E_5$ which appears in the last sum of $(\star)$ in the proof of Lemma \ref{k>4}. It is quiet obvious that the above bounds are not effective. But at the moment, we see that an analog of the aforementioned lemma can be obtained for each increasing sequence of the form $\mathcal{A}=(1,2,3,4,5,a_6,a_7,\ldots)$. In particular, it means that we are also able to prove a corresponding theorem.
\begin{thm}\label{log2}
For every increasing sequence $\mathcal{A}=(1,2,3,4,5,a_6,a_7,\ldots)\in\mathbb{N}^\infty$ and $k>5$, the inequality
\begin{align}\label{l2}
    p_\mathcal{A}^2(n,k)>\left(1+\frac{1}{n^2}\right)p_\mathcal{A}(n+1,k)p_\mathcal{A}(n-1,k)
\end{align}
holds for each $n>432k^2(k-1)!a_k^{k+3}e^\frac{1}{a_k}$. In particular, the sequence $\left(p_\mathcal{A}(n,k)\right)_{n=1}^\infty$ is log-concave for all such $n$.
\end{thm}
\begin{proof}
It is enough to repeat the proof of Theorem \ref{log1}. 
\end{proof}
\begin{cor}\label{naturalcase}
If $k\geqslant5$, then the sequence $\left(p_{\mathcal{A}_1}(n,k)\right)_{n=1}^\infty$ is log-concave for all $n>432k^{k+5}(k-1)!e^\frac{1}{k}$.
\end{cor}
\begin{proof}
It immediately follows from Theorem \ref{log2}.
\end{proof}
Step by step we realize that the sufficient condition for the $\log$-concavity of $p_\mathcal{A}(n,k)$ is nested in the body of Theorem \ref{2.4} with $k\geqslant4$ and $j=k-2$. This issue will be discussed in details in the forthcoming section.

At the end of this part, let us consider a few examples and check how our theorems work in practice. In order to make our text more transparent we introduce the following notation.

\begin{df}
Let $\mathcal{A}=\left(a_i\right)_{i=1}^\infty$ be an arbitrary sequence of positive integers. For a given integer $k\geqslant1$, we put
\begin{align}
    \Delta_{\mathcal{A},k}(n):=p_\mathcal{A}^2(n,k)-p_\mathcal{A}(n+1,k)p_\mathcal{A}(n-1,k).
\end{align}
\end{df}
It is obvious that the sequence $\left(p_\mathcal{A}(n,k)\right)_{n=1}^\infty$ is $\log$-concave if and only if $\Delta_{\mathcal{A},k}(n)$ is positive for all sufficiently large values of $n$. 
\begin{ex}\label{Example2}
Let us observe how behaves $\Delta_{\mathcal{A}_1,k}(n)$ for $\mathcal{A}_1=(1,2,3,4,5,6,\ldots)$ and $3\leqslant k\leqslant6$.
\begin{center}
\begin{figure}[!htb]
   \begin{minipage}{0.47\textwidth}
     \centering
     \includegraphics[width=.9\linewidth]{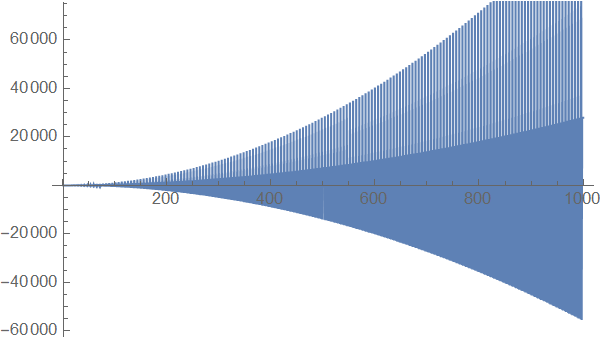}
     \caption{Values of $\Delta_{\mathcal{A}_1,3}(n)$ for $2\leqslant n \leqslant1000$}
   \end{minipage}\hfill
   \begin{minipage}{0.47\textwidth}
     \centering
     \includegraphics[width=.9\linewidth]{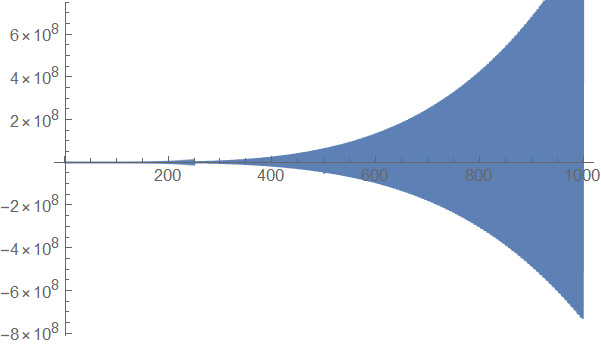}
     \caption{Values of $\Delta_{\mathcal{A}_1,4}(n)$ for $2\leqslant n \leqslant1000$}
   \end{minipage}
\end{figure}
\begin{figure}[!htb]
   \begin{minipage}{0.47\textwidth}
     \centering
     \includegraphics[width=.9\linewidth]{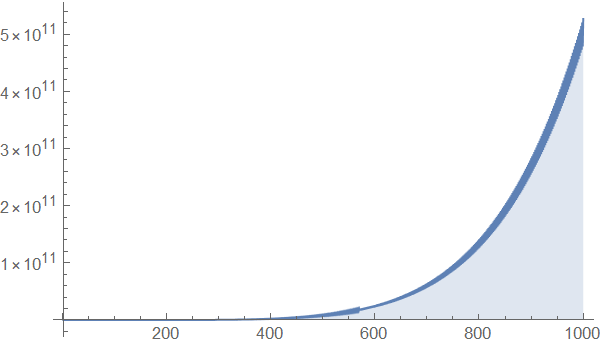}
     \caption{Values of $\Delta_{\mathcal{A}_1,5}(n)$ for $2\leqslant n \leqslant1000$}
   \end{minipage}\hfill
   \begin{minipage}{0.47\textwidth}
     \centering
     \includegraphics[width=.9\linewidth]{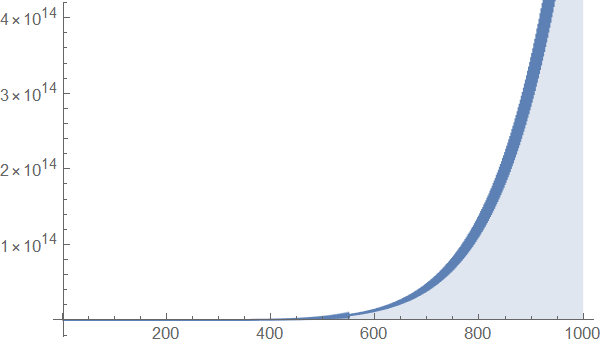}
     \caption{Values of $\Delta_{\mathcal{A}_1,6}(n)$ for $2\leqslant n \leqslant1000$}
   \end{minipage}
\end{figure}
\end{center}
The above figures agree with Corollary \ref{naturalcase} and our remarks before Proposition \ref{A_1k=5}. However, we see that the constant appearing in Corollary \ref{naturalcase} is far greater than necessary --- for instance, if $k=6$, then it requires $n>22218317100077$ while some numerical calculations in Mathematica \cite{WM} show that, in fact, $\left(p_{\mathcal{A}_1}(n,k)\right)_{n=1}^\infty$ is log-concave for every $n>79$.
\end{ex}
\begin{cor}
The sequence $\left(p_{\mathcal{A}_1}(n,6)\right)_{n=1}^\infty$ is log-concave for all $n>79$.
\end{cor}
\begin{ex}\label{Example3}
Now, we examine the restricted partition function for the sequence of consecutive prime numbers $\mathcal{P}=\left(2,3,5,7,11,\ldots\right)$ and $k\in\{2,3,4,5\}$. The graphs below agree with Proposition \ref{pr4} and Theorem \ref{log1} as well. However, once again we see that both  $\Delta_{\mathcal{P},4}(n)>0$ and $\Delta_{\mathcal{P},5}(n)>0$ hold for much more smaller numbers $n$ than Proposition \ref{pr4} and Theorem \ref{log1} require.
\begin{center}
\begin{figure}[!htb]
   \begin{minipage}{0.47\textwidth}
     \centering
     \includegraphics[width=.9\linewidth]{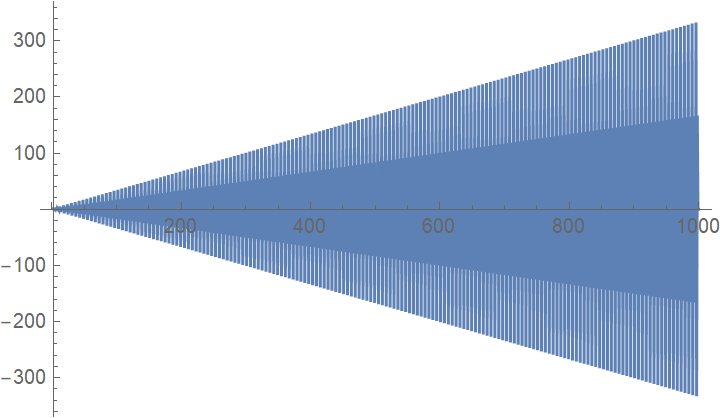}
     \caption{Values of $\Delta_{\mathcal{P},2}(n)$ for $2\leqslant n \leqslant1000$}
   \end{minipage}\hfill
   \begin{minipage}{0.47\textwidth}
     \centering
     \includegraphics[width=.9\linewidth]{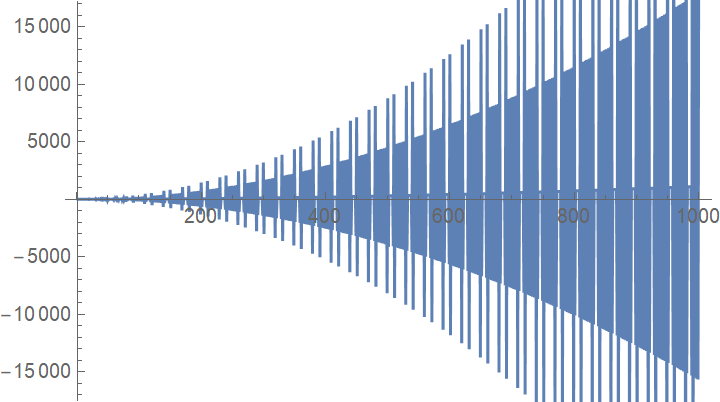}
     \caption{Values of $\Delta_{\mathcal{P},3}(n)$ for $2\leqslant n \leqslant1000$}
   \end{minipage}
\end{figure}
\begin{figure}[!htb]
   \begin{minipage}{0.47\textwidth}
     \centering
     \includegraphics[width=.9\linewidth]{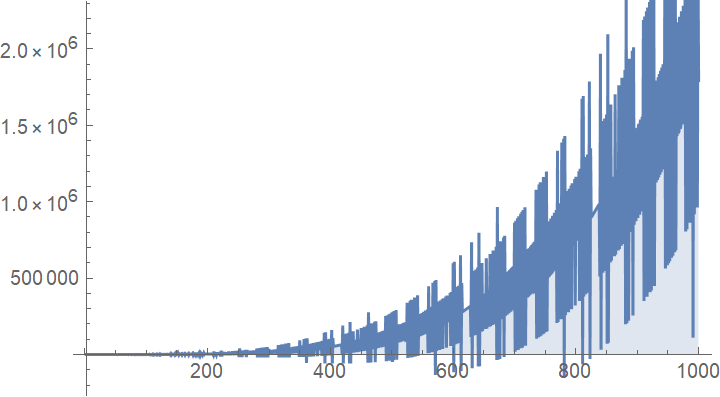}
     \caption{Values of $\Delta_{\mathcal{P},4}(n)$ for $2\leqslant n \leqslant1000$}
   \end{minipage}\hfill
   \begin{minipage}{0.47\textwidth}
     \centering
     \includegraphics[width=.9\linewidth]{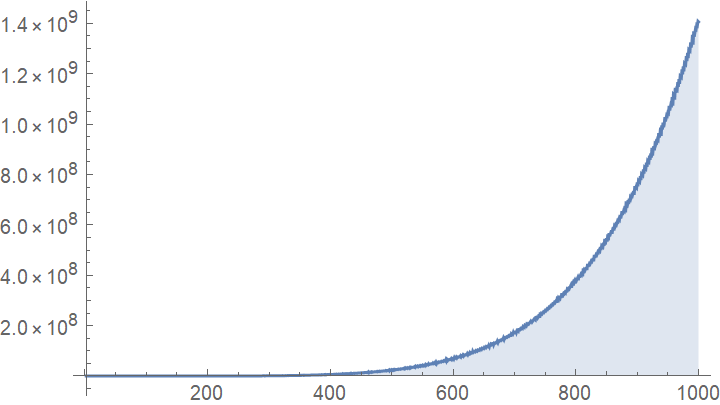}
     \caption{Values of $\Delta_{\mathcal{P},5}(n)$ for $2\leqslant n \leqslant1000$}
   \end{minipage}
\end{figure}
\end{center} 
\end{ex}
The observation we make in either Example \ref{Example2} or Example \ref{Example3} follows from the fact that we bound only the coefficient of the fourth highest degree of $p_\mathcal{A}(n,k)$. Therefore, if $k$ grows, our estimates become more and more inaccurate. Nevertheless, the main advantage of the theorems performed in this section is their universality. More precisely, there is a great wealth of choice of integer sequences, for whose one can apply them.  

The last example is devoted to an infinite family of sequences such that
for every representative $\mathcal{A}$ of this family the sequence $\left(p_\mathcal{A}(n,k)\right)_{n=1}^\infty$ is never $\log$-concave, if $k\geqslant1$.

\begin{ex}
Let $m\geqslant2$ be a fixed positive integer. We investigate the sequence of consecutive powers of $m$, namely, $\mathcal{M}=(1,m,m^2,m^3,\ldots)$. In that case $p_\mathcal{M}(n,k)$ is called the restricted $m$-ary partition function. It is not difficult to notice that if $n=sm+r$ for some non-negative integers $s$ and $r\in\{0,1,\ldots,m-1\}$, then $p_\mathcal{M}(sm+r,k)=p_\mathcal{M}(sm,k)$ for every $k\geqslant1$. Hence, if $k\geqslant2$, then $\Delta_{\mathcal{M},k}(n)<0$ if and only if $n\equiv m-1\pmod*{m}$; and, in particular, the sequence $\left(p_\mathcal{M}(n,k)\right)_{n=1}^\infty$ can not be log-concave for each $k\geqslant1$. Our observation agrees with the similar result of Ulas and Żmija for the unrestricted binary partition function \cite{MUBZ} --- that is the case of $\mathcal{M}=(1,2,4,8,\ldots)$ and $k=\infty$. For more information about both restricted and unrestricted $m$-ary partition functions we refer the reader to \cite{Gunnar, Dolph, GMU, RS1, RS2, Z}.
\end{ex}

\section{The multicolor restricted partition function }

In this section we endeavor to generalize the prior results to a quite wider family of integer sequences. Henceforth $\mathcal{A}=\left(a_i\right)_{i=1}^\infty$ denotes a non-decreasing sequence of positive integers. As in the case of the restricted partition function, the multicolor partition function counts all restricted partitions with parts in $\{a_1,a_2,\ldots,a_k\}$. The difference is that $a_i$ can be equal to $a_j$ for some $1\leqslant i<j\leqslant k$; but we want to distinguish such elements in some way, therefore we might assign a unique color to each of them --- in other words we treat $\{a_1,a_2,\ldots,a_k\}$ as a multiset. Since the multicolor restricted partition function is, in fact, an extension of the restricted partition function, it is also denoted by $p_\mathcal{A}(n,k)$. Let us illustrate the introduced definition in practise.

\begin{ex}[Restricted plane partitions]\label{Ex5.1}
Let  $\mathcal{A}=(1,2,\textcolor{blue}{2},3,\textcolor{blue}{3},\textcolor{red}{3},\ldots)$ be the sequence of consecutive positive integers such that every number $j$ appears in $j$ distinct colors. For instance, there are $8$ restricted partitions of $4$ with parts in $\{1,2,\textcolor{blue}{2},3,\textcolor{blue}{3}\}$, namely: $(\textcolor{blue}{3},1)$, $(3,1)$, $(\textcolor{blue}{2},\textcolor{blue}{2})$, $(\textcolor{blue}{2},2)$, $(2,2)$, $(\textcolor{blue}{2},1,1)$, $(2,1,1)$ and $(1,1,1,1)$. Thus, $p_\mathcal{A}(4,5)=8$. 

It is worth mentioning that if we allow $k=\infty$, then we deal with so-called plane partitions and the plane partition function, for more information see \cite[Chapter~11]{GA2} or \cite[Chapter~10]{GA1}.
\end{ex}
    
Now, we proceed to discover a Bessenrodt-Ono type inequality and a $\log$-concavity criterion for the multicolor restricted partition function. In order to do that we use a very important fact contained in the following.

\begin{re}\label{Remark5.1}
{\rm All results from Sec. 2 as well as Theorem \ref{CNT} remain valid, if we replace the condition on $\mathcal{A}$: `an increasing sequence' by `a non-decreasing sequence', and the phrases `set' and `subset' by `multiset' and `multisubset', respectively. To observe the phenomena it is just enough to go through the proofs of these issues.
}
\end{re}

Further, we derive analogous bounds to those obtained in Lemma \ref{3.3} and Lemma \ref{k>4}, however we perform a bit more subtle approach which omits a lot of technicalities and the induction steps as well. In order to achieve this goal we need to recall the well-known fact that Stirling numbers of the first kind may be defined by the coefficients of the rising factorial (Pochhammer function), namely
\begin{align}\label{risingfactorial}
    x^{\overline{n}}:=x(x+1)\cdots(x+n-1)=\sum_{i=0}^{n}\stirlingone{n}{i}x^i,
\end{align}
where $n$ is an arbitrary non-negative integer (see, \cite[Chapter~6]{GKP}).

\begin{lm}\label{Lemma5}
Let $\mathcal{A}=\left(a_i\right)_{i=1}^\infty$ be a non-decreasing sequence of positive integers, and let $k\in\mathbb{N}_{\geqslant2}$ be fixed. For a given integer $j\in\{1,2,\ldots,k\}$, if $\gcd A=1$ for all $j$-multisubsets $A$ of $\{a_1,a_2,\ldots,a_k\}$, then
\begin{equation*}
c_{k-1}n^{k-1}+\cdots+c_{j-1}n^{j-1}-Fn^{j-2}<p_\mathcal{A}(n,k)<c_{k-1}n^{k-1}+\cdots+c_{j-1}n^{j-1}+Fn^{j-2}
\end{equation*}
holds for every $n>0$, where all the coefficients $c_i$ are uniquely determined by Theorem \ref{2.4}, and $F=\frac{\prod_{i=1}^k(1+iDk)}{k!\prod_{i=1}^ka_i}$ with $D=\lcm(a_1,a_2,\ldots,a_k)$.
\end{lm}
\begin{proof}
Theorem \ref{2.3} together with Remark \ref{Remark1} and Remark \ref{Remark5.1} maintain that
\begin{align*}
    p_\mathcal{A}(n,k)=c_{k-1}n^{k-1}+c_{k-2}n^{k-2}+\cdots+c_{j-1}n^{j-1}+c_{j-2}(n)n^{j-2}+\cdots+c_0(n),
\end{align*}
where $c_{j-1},c_j,\ldots,c_{k-1}$ are fixed, while $c_0(n),c_1(n),\ldots,c_{j-2}(n)$ depend on the residue class of $n\pmod*{D}$ with $D=\lcm(a_1,a_2,\ldots,a_k)$. Now, we simply use Theorem \ref{CNT} in order to estimate, let say, the quasi-polynomial part of $p_\mathcal{A}(n,k)$ from above and below as well. For every $0 \leqslant m<j-1$, we have
\begin{align*}
    |c_m(n)|=&\bigg{|}\frac{1}{(k-1)!}\sum_{\substack{ 0\leqslant j_1\leqslant\frac{D}{a_1}-1,\ldots,0\leqslant j_k\leqslant\frac{D}{a_k}-1\\
    a_1j_1+\cdots+a_kj_k\equiv n\pmod*{D}}}\sum_{i=m}^{k-1}\stirlingone{k}{i}(-1)^{i-m}\binom{i}{m}\\
    &\times D^{-i}(a_1j_1+\cdots+a_kj_k)^{i-m}\bigg{|}<\frac{D^{k}}{(k-1)!\prod_{j=1}^ka_j}\sum_{i=m}^{k-1}\stirlingone{k}{i}\binom{i}{m}\frac{(kD)^{i-m}}{D^i}\\
    \leqslant&\frac{D^{k-m}k^{k-1-m}}{(k-1)!\prod_{j=1}^ka_j}\sum_{i=m}^{k-1}\stirlingone{k}{i}\binom{i}{m}<\frac{D^{k-m}k^{k-1-m}}{(k-1)!\prod_{j=1}^ka_j}\stirlingone{k+1}{m+1},
\end{align*}
where the last inequality follows from the identity
\begin{align*}
    \sum_{i=m}^k\stirlingone{k}{i}\binom{i}{m}=\stirlingone{k+1}{m+1},
\end{align*}
which might be found in \cite[Chapter~6]{GKP}. Therefore, since $n\geqslant1$ we have
\begin{align*}
    \left|\sum_{m=0}^{j-2}c_m(n)n^m\right|&<\sum_{m=0}^{j-2}\frac{D^{k-m}k^{k-1-m}}{(k-1)!\prod_{i=1}^ka_i}\stirlingone{k+1}{m+1}n^m\\
    &\leqslant\frac{D^{k+1}k^{k}n^{j-2}}{(k-1)!\prod_{i=1}^ka_i}\sum_{m=0}^{j-2}\stirlingone{k+1}{m+1}(Dk)^{-1-m}\\
    &<\frac{D^{k+1}k^{k}n^{j-2}}{(k-1)!\prod_{i=1}^ka_i}\sum_{m=0}^{k+1}\stirlingone{k+1}{m}(Dk)^{-m}\\
    &=\frac{D^{k+1}k^{k}n^{j-2}}{(k-1)!\prod_{i=1}^ka_i}\left(\frac{1}{Dk}\right)^{\overline{k+1}}\\
    &=\frac{\prod_{i=1}^k(1+iDk)}{k!\prod_{i=1}^ka_i}n^{j-2},
\end{align*}
where the penultimate line is the consequence of (\ref{risingfactorial}). Hence, the proof is complete.
\end{proof}

Before we apply the above lemma and deduce a generalization of the inequality from Theorem \ref{thmbo}, let us note that Corollary \ref{cor2} also remains valid for the multicolor partition function. Indeed, if $k=2$, then there is only one additional possibility which is not taken into account by the statement, namely a sequence $\mathcal{A}$ such that $a_1=a_2=1$, but in this case $p_\mathcal{A}(n,2)=n+1$; and it might be easily checked that
\begin{align*}
    p_\mathcal{A}(a,2)p_\mathcal{A}(b,2)>p_\mathcal{A}(a+b,2)
\end{align*} holds for all positive numbers $a$ and $b$.
\begin{thm}\label{Thm5.1}
Let $k\geqslant2$ be fixed, and let $\mathcal{A}=\left(a_i\right)_{i=1}^\infty$ be a non-decreasing sequence of positive integers such that $\gcd(a_1,a_2,\ldots,a_k)=1$. For all $a,b>\frac{2\prod_{i=1}^k(1+iDk)}{k}+2$,
we have
\begin{align*}
    p_\mathcal{A}(a,k)p_\mathcal{A}(b,k)>p_\mathcal{A}(a+b,k).
\end{align*}
\end{thm}
\begin{proof}
If $k=2$, then the statement is clear by the preceding comment and Corollary \ref{cor2}. Hence, we assume that $k\geqslant3$. By Theorem \ref{2.3} and Lemma \ref{Lemma5}, it is enough to check under what conditions on $a$ and $b$, the inequality  
\begin{align*}
    \left(\frac{a^{k-1}}{(k-1)!\prod_{i=1}^ka_i}-Fa^{k-2}\right)&\left(\frac{b^{k-1}}{(k-1)!\prod_{i=1}^ka_i}-Fb^{k-2}\right)\\
    &>\frac{(a+b)^{k-1}}{(k-1)!\prod_{i=1}^ka_i}+F(a+b)^{k-2}
\end{align*}
is satisfied, where $F=\frac{\prod_{i=1}^k(1+iDk)}{k!\prod_{i=1}^ka_i}$ with $D=\lcm(a_1,a_2,\ldots,a_k)$. Next, the above may be simplified to
\begin{align*}
    \frac{1}{(k-1)!\prod_{i=1}^ka_i}\left(\frac{ab}{a+b}\right)^{k-2}(a-c)(b-c)>(a+b+c),
\end{align*}
where $c=\frac{\prod_{i=1}^k(1+iDk)}{k}$. Without loss of generality, let us assume that $b\leqslant a$. Similarly to the proof of Theorem \ref{thmbo}, it is enough to examine when the inequalities:
\begin{align}
    \left(\frac{ab}{a+b}\right)^{k-2}&\geqslant (k-1)!\prod_{i=1}^ka_i\label{5.2}\\
    (a-c)(b-c)&>(a+b+c)\label{5.3}
\end{align}
are true. Since $b\leqslant a$ and $k\geqslant3$, one can reduce (\ref{5.2}) to 
\begin{align*}
    b\geqslant2\sqrt[k-2]{(k-1)!\prod_{i=1}^ka_i}.
\end{align*}
 In the case of (\ref{5.3}), on the other hand, we might get
 \begin{align*}
    ab-2a(c+1)+c(c-1)>0;
\end{align*}
and it suffices to take $b\geqslant2(c+1)$ as $c>1$. However, let us observe that 
\begin{align*}
    2\sqrt[k-2]{(k-1)!\prod_{i=1}^ka_i}\leqslant2(k-1)!\prod_{i=1}^ka_i
\end{align*}
and
\begin{align*}
    2c+2=\frac{2}{k}(1+Dk)\cdots(1+Dk^2)+2>\frac{2}{k}(Dk)(2Dk)\cdots(kDk)>2(k-1)!\prod_{i=1}^ka_i.
\end{align*}
Therefore, it is enough to assume that $a,b\geqslant\frac{2\prod_{i=1}^k(1+iDk)}{k}+2$, as required.
\end{proof}

It is quite unfortunate that our assumptions on the values of $a$ and $b$ in the above are so strong. Nevertheless, the main advantage of the theorem is its universality. Furthermore, we can not extend the Bessenrodt-Ono inequality for $p_\mathcal{A}(n,k)$ even more, because of the following.

\begin{cor}\label{cor3}
Let $\mathcal{A}=\left(a_i\right)_{i=1}^\infty$ be a non-decreasing sequence of positive integers. The Bessenrodt-Ono inequality for $p_\mathcal{A}(n,k)$ occurs if and only if $k\geqslant2$ and $\gcd(a_1,a_2,\ldots,a_k)=1$.
\end{cor}
\begin{proof}
The implication from the right to left is straightforward by Theorem \ref{Thm5.1}. For the converse, suppose to the contrary that $k=1$ or $\gcd(a_1,a_2,\ldots,a_k)>1$. If $k=1$, then the Bessenrodt-Ono inequality for $p_\mathcal{A}(n,1)$ can not hold by Proposition \ref{Pr2.1}. On the other hand, if $d:=\gcd(a_1,a_2,\ldots,a_k)>1$ and $a$ or $b$ is not divisible by $d$, then
\begin{align*}
    0=p_\mathcal{A}(a,k)p_\mathcal{A}(b,k)\leqslant p_\mathcal{A}(a+b,k);
\end{align*}
and the proof is complete.
\end{proof}

In other words the corollary says that Bessenrodt-Ono inequality for $p_\mathcal{A}(n,k)$ needs to be fulfilled for all reasonable sequences $\mathcal{A}$ and multisets $\{a_1,a_2,\ldots,a_k\}$. 

The next part of this section is devoted to investigating the $\log$-concavity of the multicolor partition function and enhancing both Theorem \ref{log1} and Theorem \ref{log2}.

\begin{thm}\label{log3}
Let $\mathcal{A}=\left(a_i\right)_{i=1}^\infty$ be a non-decreasing sequence of positive integers. If $1<k<4$ and $a_1=\cdots=a_k=1$, or $k\geqslant4$ and $\gcd A=1$ for all $(k-2)$-multisubsets $A$ of $\{a_1,a_2,\ldots,a_k\}$, then
\begin{align}\label{logi1}
    p_\mathcal{A}^2(n,k)>p_\mathcal{A}(n+1,k)p_\mathcal{A}(n-1,k)
\end{align}
for every $n\geqslant\frac{2\prod_{i=1}^k(1+iDk)}{k},$ where $D=\lcm{(a_1,a_2,\ldots,a_k)}$. Moreover, if $k>4$, then
\begin{align}\label{logi2}
    p_\mathcal{A}^2(n,k)>\left(1+\frac{1}{n^2}\right)p_\mathcal{A}(n+1,k)p_\mathcal{A}(n-1,k).
\end{align}
holds for all $n\geqslant\frac{2\prod_{i=1}^k(1+iDk)}{k}$. For $k=4$, (\ref{logi2}) is true for each $n\geqslant\frac{3\prod_{i=1}^k(1+iDk)}{k}$. Additionally, for the constant sequence $\mathcal{A}=(1,\textcolor{blue}{1},\textcolor{red}{1},\ldots)$, we have that (\ref{logi1}) is satisfied for all positive integers $n$ and $k\geqslant2$; and (\ref{logi2}) is true for any integers $k\geqslant3$ and $n\geqslant\frac{k}{k-2}$. 
\end{thm}
\begin{proof}
At first, let $\mathcal{A}=(1,\textcolor{blue}{1},\textcolor{red}{1},\ldots)$ be a constant sequence. It is known that $$p_\mathcal{A}(n,k)=\binom{n+k-1}{k-1},$$ and is easy to show that (\ref{logi1}) is satisfied for all $k\geqslant2$ and $n\geqslant1$ as well as (\ref{logi2}) is valid for each $k\geqslant3$ and $n\geqslant\frac{k}{k-2}$. 

Further, let $k\geqslant4$. We assume that $\mathcal{A}$ is any non-decreasing sequence of positive integers such that $\gcd A=1$ for all $(k-2)$-multisubsets $A$ of $\{a_1,a_2,\ldots,a_k\}$. Theorem \ref{2.4}, Remark \ref{Remark5.1} and Lemma \ref{Lemma5} assert that
\begin{align*}
    \alpha n^{k-1}+\beta n^{k-2}+\gamma n^{k-3}-Fn^{k-4}<p_\mathcal{A}(n,k)<\alpha n^{k-1}+\beta n^{k-2}+\gamma n^{k-3}+Fn^{k-4},
\end{align*}
where 
\begin{align*}
    \alpha&=\frac{1}{(k-1)!\prod_{i=1}^ka_i},\\
    \beta&=\frac{\sigma}{2(k-2)!\prod_{i=1}^ka_i},\\
    \gamma&=\frac{3\sigma^2-s_2}{24(k-3)!\prod_{i=1}^ka_i},\\
    F&=\frac{\prod_{i=1}^k(1+iDk)}{k!\prod_{i=1}^ka_i}
\end{align*}
with $\sigma=a_1+a_2+\cdots+a_k$, $s_2=a_1^2+a_2^2+\cdots+a_k^2$ and $D=\lcm(a_1,a_2,\ldots,a_k)$. Let us first examine under what conditions on $n$, (\ref{logi2}) holds. Similarly to the proofs of Proposition \ref{pr4} and Theorem \ref{log1}, it suffices to consider the inequality:
\begin{align*}
    &[\alpha n^{k-1}+\beta n^{k-2}+\gamma n^{k-3}-Fn^{k-4}]\times[\alpha n^{k-1}+\beta n^{k-2}+\gamma n^{k-3}-Fn^{k-4}]\\
    &>\left(1+\frac{1}{n^2}\right)[\alpha (n+1)^{k-1}+\beta (n+1)^{k-2}+\gamma (n+1)^{k-3}+F(n+1)^{k-4}]\\
    &\phantom{>.\left(1\frac{1}{n^2}\right)}\times[\alpha (n-1)^{k-1}+\beta (n-1)^{k-2}+\gamma (n-1)^{k-3}+F(n-1)^{k-4}].
\end{align*}
A similar approach works also in the case of (\ref{logi1}). Since the reasoning is very technical and, actually, the same as in Proposition \ref{pr4} and Theorem \ref{log1}, we omit all the remaining details.
\end{proof}

One can ask whether there is a similar result to Corollary \ref{cor3} in the case of the $\log$-concavity of $p_\mathcal{A}(n,k)$. Indeed, but we need little preparation in order to show that the answer is positive. 

The first issue is related to Proposition \ref{Pr2.1}, namely, let us observe that the recurrence $p_\mathcal{A}(n,k)$ also holds if we replace $a_k$ by any other number $a_i$ for $1\leqslant i\leqslant k$. However, there appears a confusing part --- $p_\mathcal{A}(n,k-1)$. For the sake of clarity, let $\mathcal{A}_{(j,k)}$ denote a permutation of $\mathcal{A}$ such that $a_j$ swaps places with $a_k$. Thus, we might reformulate the equality (\ref{Rf}) to 
\begin{align}
    p_{\mathcal{A}_{(i,k)}}(n,k)=p_{\mathcal{A}_{(i,k)}}(n-a_{i},k)+p_{\mathcal{A}_{(i,k)}}(n,k-1),
\end{align}
where $i\in\{1,2,\ldots,k\}$ is arbitrary.

Now, we are able to show the converse implication to the one from Theorem \ref{2.4}.

\begin{pr}\label{pr5.7}
Let $\mathcal{A}=\left(a_i\right)_{i=1}^\infty$ be a non-decreasing sequence of positive integers, and let $k\in\mathbb{N}_{+}$ be fixed. If
\begin{equation*}
p_\mathcal{A}(n,k)=c_{k-1}n^{k-1}+c_{k-2}n^{k-2}+\cdots+c_{j-1}n^{j-1}+c_{j-2}(n)n^{j-2}+\cdots+c_0(n),
\end{equation*}
where $c_{k-1},\ldots,c_{j-1}$ are independent of a residue class $n\pmod*{\lcm(a_1,\ldots,a_k)}$, then $\gcd A=1$ for all $j$-multisubsets $A$ of $\{a_1,a_2,\ldots,a_k\}$.
\end{pr}
\begin{proof}
For every $k\geqslant1$, if $j=k$, then it is clear that $\gcd(a_1,\ldots,a_k)=1$, otherwise $p_\mathcal{A}(n,k)=0$ as well as $p_\mathcal{A}(m,k)>0$ for infinitely many values of $n$ and $m$, respectively, and the leading coefficient of $p_\mathcal{A}(n,k)$ can not be fixed.

Let us fix $j\in\mathbb{N}_+$, suppose that the claim holds for each $k=j,j+1,\ldots,l-1$ and check whether it is also valid for $k=l$. By the preparation before the statement, for arbitrary chosen parts $a_{i_1},a_{i_2},\ldots,a_{i_j}\in\{a_1,a_2,\ldots,a_l\}$, we have
\begin{align*}
    p_\mathcal{A}(n,l)=p_{\mathcal{A}_{(i_s,l)}}(n,l)=p_{\mathcal{A}_{(i_s,l)}}(n-a_{i_s},l)+p_{\mathcal{A}_{(i_s,l)}}(n,l-1)
\end{align*}
for all $s\in\{1,2,\ldots,j\}$. Since the coefficients $c_{l-1},c_{l-2},\ldots,c_{j-1}$ of $p_\mathcal{A}(n,l)$ are independent of a residue class $n\pmod*{\lcm(a_1,\ldots,a_l)}$, we also see that the corresponding coefficients, say $d_{l-2,s},d_{l-3,s},\ldots,d_{j-1,s}$, of
\begin{align*}
    p_{\mathcal{A}_{(i_s,l)}}(n,l-1)=d_{l-2,s}n^{l-2}+\cdots+d_{j-1,s}n^{j-1}+d_{j-2,s}(n)n^{j-2}+\cdots+d_{0,s}(n)
\end{align*}
do not depend on a residue class $n\pmod*{\lcm(a_1,\ldots,a_l)}$ for every $1\leqslant s\leqslant j$. For each such a number $s$, the induction hypothesis maintains that $\gcd{A_s}=1$ for all $j$-multisubsets $A_s$ of $\{a_1,\ldots,a_l\}\setminus\{a_{i_s}\}$. Therefore, it is enough to verify that $\gcd{(a_{i_1},a_{i_2},\ldots,a_{i_j})}=1$. But since $j<l$, one might also find an element $a_t\in\{a_1,\ldots,a_l\}\setminus\{a_{i_1},a_{i_2},\ldots,a_{i_j}\}$, replace its role with $a_{l}$, repeat the above presented reasoning and deduce the desired equality. Finally, the proof is complete by the law of induction.

\end{proof}

We are ready to show an analogue of Corollary \ref{cor3} now.  

\begin{cor}\label{cor4}
Let $\mathcal{A}=\left(a_i\right)_{i=1}^\infty$ be a non-decreasing sequence of positive integers. The sequence $\left(p_\mathcal{A}(n,k)\right)_{n=1}^\infty$ is eventually log-concave if and only if $k\geqslant2$ and $a_1=\cdots=a_k=1$ or $k\geqslant4$ and $\gcd A=1$ for all $(k-2)$-multisubsets $A$ of $\{a_1,a_2,\ldots,a_k\}$.
\end{cor}

\begin{proof} 
The implication to the left is obvious by Theorem \ref{log3}. To prove the implication to the right let us fix $k\in\mathbb{N}_+$. It is clear that the sequence $\left(p_\mathcal{A}(n,1)\right)_{n=1}^\infty$ can not be $\log$-concave. Thus, let $k\geqslant2$ and suppose, for contradiction, that the assumptions on the numbers $a_1,a_2,\ldots,a_k$ do not hold. By Proposition \ref{pr5.7}, we have that 
\begin{align*}
    p_\mathcal{A}(n,k)=c_{k-1}(n)n^{k-1}+c_{k-2}(n)n^{k-2}+c_{k-3}(n)n^{k-3}+\cdots+c_0(n),
\end{align*}
where at least one of the coefficients: $c_{k-1}(n)$, $c_{k-2}(n)$ or $c_{k-3}(n)$ depends on \linebreak a residue class of $n\pmod*{\lcm(a_1,\ldots,a_k)}$. Let $t\in\{k-3,k-2,k-1\}$ be the greatest index of this property. Next, it suffices to take any $n\pmod*{\lcm(a_1,\ldots,a_k)}$ such that $c_t(n)$ is the smallest, and at least one of $c_t(n+1)$ or $c_t(n-1)$ is strictly larger than $c_t(n)$. If we do so, then it turns out that the leading coefficient of
\begin{align*}
    p_\mathcal{A}^2(n,k)-p_\mathcal{A}(n+1,k)p_\mathcal{A}(n-1,k)
\end{align*}
is negative, and the sequence $\left(p_\mathcal{A}(n,k)\right)_{n=1}^\infty$ can not be $\log$-concave, as required.
\end{proof}

At the end of the sequel, let us go back to Example \ref{Ex5.1} and observe how our $\log$-concavity criterion works in this case.

\begin{ex}
For the sequence $\mathcal{A}=(1,2,\textcolor{blue}{2},3,\textcolor{blue}{3},\textcolor{red}{3},\ldots)$ defined in Example \ref{Ex5.1}, the behavior of the function $\Delta_{\mathcal{A},k}(n)=p_\mathcal{A}^2(n,k)-p_\mathcal{A}(n+1,k)p_\mathcal{A}(n-1,k)$ for $3\leqslant k\leqslant6$ presents as follows:
\begin{center}
\begin{figure}[!htb]
   \begin{minipage}{0.47\textwidth}
     \centering
     \includegraphics[width=.9\linewidth]{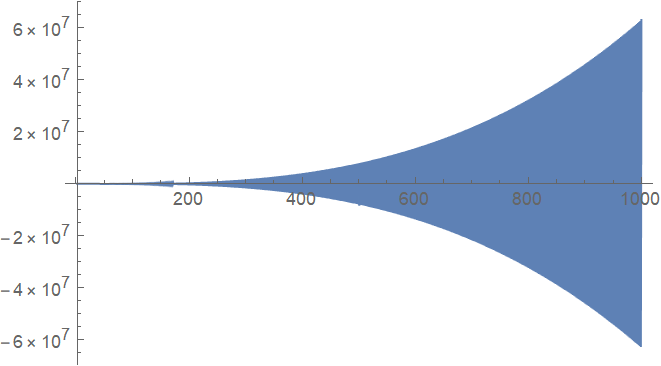}
     \caption{Values of $\Delta_{\mathcal{A},3}(n)$ for $2\leqslant n \leqslant1000$}
   \end{minipage}\hfill
   \begin{minipage}{0.47\textwidth}
     \centering
     \includegraphics[width=.9\linewidth]{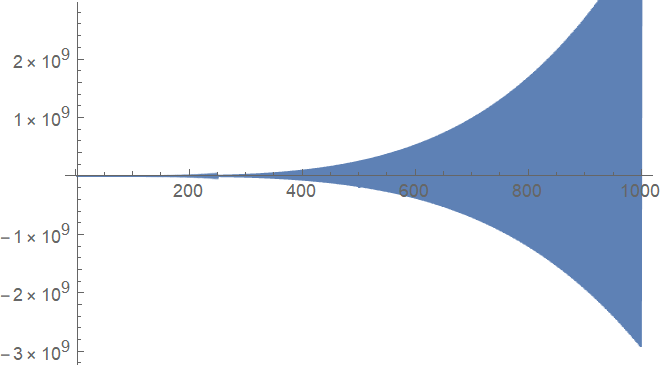}
     \caption{Values of $\Delta_{\mathcal{A},4}(n)$ for $2\leqslant n \leqslant1000$}
   \end{minipage}
\end{figure}
\begin{figure}[!htb]
   \begin{minipage}{0.47\textwidth}
     \centering
     \includegraphics[width=.9\linewidth]{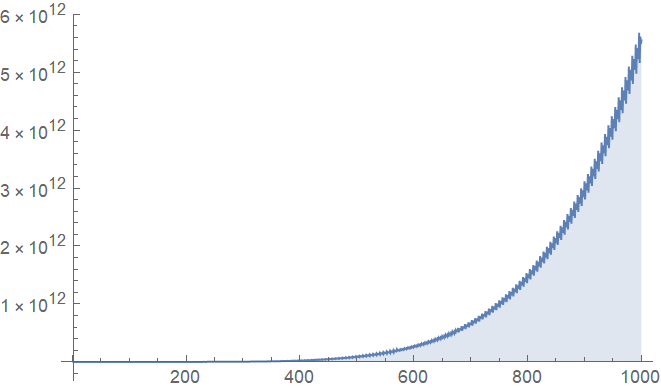}
     \caption{Values of $\Delta_{\mathcal{A},5}(n)$ for $2\leqslant n \leqslant1000$}
   \end{minipage}\hfill
   \begin{minipage}{0.47\textwidth}
     \centering
     \includegraphics[width=.9\linewidth]{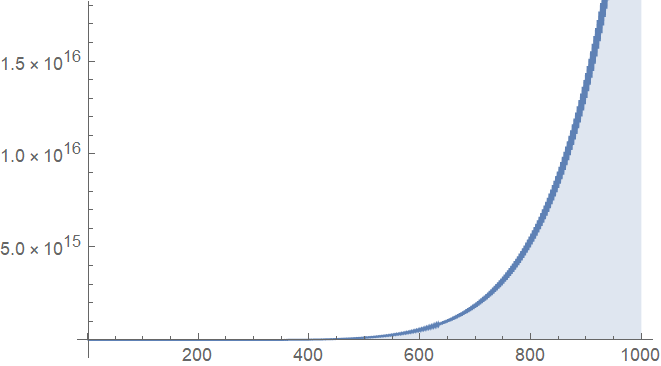}
     \caption{Values of $\Delta_{\mathcal{A},6}(n)$ for $2\leqslant n \leqslant1000$}
   \end{minipage}
\end{figure}
\end{center}
These graphs agree with our results from this section. However, once again we observe that the requirements for $n$ in Theorem \ref{log3} are not the best ones.
\end{ex}

\section*{Acknowledgements}
I would like to thank Maciej Ulas and Piotr Miska for their time, effort and profound comments. I am also grateful to Christophe Vignat for his helpful suggestion. This research was funded by both a grant of the National Science Centre (NCN), Poland, no. UMO-2019/34/E/ST1/00094 and by the Priority Research Area SciMat under the program Excellence Initiative – Research University at the Jagiellonian University in Kraków.

\end{document}